\documentclass[a4paper,12pt]{article}
\usepackage{amsmath,amsthm,amssymb,latexsym}
\title{{\bf Capacities associated with scalar signed Riesz kernels, and analytic capacity}}
\author{\Large{\Large Joan Mateu, Laura Prat and Joan Verdera}}
\date{}
\newcommand{\guio}[1]{\nobreakdash-\hspace{0pt}#1}

\begin{document}

\maketitle
\newtheorem{teo}{Theorem}
\newtheorem*{teo*}{Theorem}
\newtheorem{co}[teo]{Corollary}
\newtheorem{lemma}[teo]{Lemma}
\newtheorem{prop}[teo]{Proposition}
\newtheorem*{sublema*}{Sublemma}

\theoremstyle{definition}
\newtheorem{defi}[teo]{Definition}
\newtheorem*{gracies}{Acknowledgements}

\theoremstyle{remark}
\newtheorem*{note}{Note}
\newtheorem*{remark}{Remark}

\newcommand{\Ha}{{\mathcal H}^{\alpha}}
\newcommand{\Hu}{{\mathcal H}^1}
\newcommand{\Rn}{{\mathbb R}^n}
\newcommand{\Rd}{{\mathbb R}^2}
\newcommand{\ep}{\varepsilon}
\newcommand{\N}{\mathbb{N}}
\newcommand{\Z}{\mathbb{Z}}
\newcommand{\C}{\mathbb{C}}
\newcommand{\op}{\operatorname{op}}
\newcommand{\cc}{{\mathcal C}}

\begin{abstract}
Analytic capacity is associated with the Cauchy kernel $1/z$ and the
space $L^\infty$. One has likewise capacities associated with the
real and imaginary parts of the Cauchy kernel and $L^\infty$.
Striking results of Tolsa and a simple remark show that these three
capacities are comparable. We present an extension of this fact to
$\Rn$, $n\geq 3$, involving the vector valued Riesz kernel of
homogeneity $-1$ and $n-1$ of its components.
\end{abstract}

\section{Introduction}
The analytic capacity of a compact subset $E$ of the plane is
defined by
$$ \gamma(E)=\sup|f'(\infty)|$$
where the supremum is taken over those analytic functions on
$\mathbb{C}\setminus E$ such that $|f(z)|\leq 1$,  $z\in
\mathbb{C}\setminus E$.  Sets of zero analytic capacity are exactly
the removable sets for bounded analytic functions, as it is easily
seen, and thus $\gamma(E)$ quantifies the non-removability of $E$.
Early work on analytic capacity used basically one complex variable
methods (see, e.g., \cite{Ahlfors}, \cite{garnett} and \cite{Vi}).
Analytic capacity may be written as
\begin{equation}\label{acap2}
\gamma(E)=\sup |\langle T,1\rangle|
\end{equation}
where the supremum is taken over all complex distributions $T$
supported on $E$ whose Cauchy potential $f=1/z * T$ is in the closed
unit ball of  $L^\infty(\mathbb{C})$. The transition from~$f$ to~$T$
and viceversa is performed through the formulae
$\displaystyle{T=\frac 1\pi\overline{\partial} f}$ and $f= 1/z * T$.


Expression \eqref{acap2} shows that analytic capacity is formally an
analogue of classical logarithmic capacity, in which the logarithmic
kernel has been replaced by the complex kernel $1/z$. This suggests
that real variables techniques could help in studying analytic
capacity, in spite of the fact that the Cauchy kernel is complex. In
fact, significant progress in the understanding of analytic capacity
was achieved when real variables methods, in particular the
Calder\'{o}n-Zygmund theory of the Cauchy singular integral, were
systematically used (\cite{C}, \cite{davidvitushkin},  \cite{mmv},
\cite{mtv}, \cite{semiad} and~\cite{bilipschitz}). A striking result
of Tolsa \cite{semiad} asserts that analytic capacity is comparable
to a smaller quantity, called positive analytic capacity, which is
defined on compact sets $E$ by
$$\gamma_+(E)=\sup \mu(E)$$
where the supremum is taken over those positive measures supported
on $E$ whose Cauchy potential $1/z * \mu$ is in the closed unit ball
of $L^\infty(\mathbb{C})$. In other words, there exists a positive
constant $C$ such that
\begin{equation}\label{tolsa}
\gamma(E) \le C\, \gamma_+(E),
\end{equation}
for each compact subset $E$ of the plane. This implies, in
particular,  that analytic capacity is comparable to planar
Lipschitz harmonic capacity. The Lipschitz harmonic capacity of a
compact subset of $\Rn$ is defined by
\begin{equation}\label{LipCapAlt}
\kappa(E)= \sup|\langle T,1\rangle|
\end{equation}
where the supremum is taken over those real distributions $T$
supported on $E$ such that the vector field
$\displaystyle{\frac{x}{|x|^n}*T}$ is in the unit ball of
$L^\infty(\Rn, \Rn)$. The terminology stems from the fact that
$\kappa(E)$ vanishes if and only if $E$ is removable for harmonic
functions on $\Rn\setminus E$ satisfying a global Lipschitz
condition.
 Notice that the fact that analytic capacity and Lipschitz harmonic capacity in
 the plane are comparable cannot be deduced
 just by inspection from \eqref{acap2} and \eqref{LipCapAlt}. The reason is that
 the distributions involved in the supremum in  \eqref{acap2} are
 complex.

  For a compact subset $E$ of $\Rn$ and $1 \leq i
\leq n$ set
\begin{equation}\label{kappascalar}
\kappa_i (E)= \sup|\langle T,1\rangle|
\end{equation}
where the supremum is taken over those real distributions $T$ such
that the scalar signed $i$-th Riesz potential
\begin{equation}\label{scalarpotential}
\frac{x_i}{|x|^2}*T
\end{equation}
is in the unit ball of $L^\infty(\Rn)$.

In the plane, in spite of what has been said before, it is precisely
a simple complex analytic argument that provides a complete
characterization of  the capacities $\kappa_1$ and $\kappa_2$. For
some positive constant $C$ and for each compact subset $E$ of the
plane, we have

\begin{equation}\label{mark}
C^{-1}\,\kappa_i(E)\le \gamma(E)\le C\,\kappa_i(E),\,\,i=1,2.
\end{equation}
Indeed, if $T$ is a real distribution supported on $E$ such that
$\displaystyle{\frac{x_1}{|x|^2}*T}$ is in the unit ball of
$L^\infty(\Rd)$, then $\displaystyle{\frac{1}{z}*T}$ is an analytic
function on $\C\setminus E$ whose real part is bounded in absolute
value by $1$. Mapping conformally the strip $\{z\in
\C\,:\,|Re(z)|\le 1\}$ onto the unit disk we get a function $f$,
bounded and analytic on $\C\setminus E$, such that $|\langle T,1
\rangle|\le C\,|f'(\infty)|$ (\cite{gamelin}). Hence
$C^{-1}\,\kappa_i(E)\le \gamma(E).$  The second inequality in
\eqref{mark} is an immediate consequence of the striking inequality
\eqref{tolsa}, because the real part of the Cauchy potential of a
positive measure $\mu$ is precisely $ {x_1 /|x|^2} * \mu. $ Notice
that this is not the case if $\mu$ is a complex measure.


Although there are obvious
 formal similarities between the definitions of the set functions in \eqref{acap2}
 and \eqref{kappascalar}, very little is known
about $\kappa_i$ for $n\ge 3$. The reader will find in section
\ref{finiteness} a proof of the elementary fact that $\kappa_i(E)$
is finite for each compact subset $E$ of $\Rn$. The reason why
$\kappa_i$ is difficult to understand in higher dimensions is that
boundedness of the potential \eqref{scalarpotential} does not
provide any linear growth condition on~$T$ in dimensions $n \geq 2$
(then, even in dimension $2$). Concretely, it is not true that
boundedness of \eqref{scalarpotential} implies that for each cube
$Q$ one has
\begin{equation}\label{creixement}
|\langle T,\varphi_Q\rangle |\leq C\,l(Q),
\end{equation}
for each test function $\varphi_Q \in {\mathcal C}^\infty_0(Q)$
satisfying $\|\partial^s \varphi_Q \|_\infty \le l(Q)^{-|s|}$ for
all multi-indexes $s$ of length not greater than some positive
integer $N$. Here $l(Q)$ stands for the side length of $Q$ and we
are adopting the standard notation related to multi-indexes, that
is, $s=(s_1,\dotsc,s_n)$, where each coordinate $s_j$ is a
non-negative integer and $|s|= s_1+\dotsb+s_n$ is the length of $s$.
The reader will find in section 5 three exemples of such phenomenon.
The fact that this examples exist also in dimension $2$, makes the
first inequality in (\ref{mark}) very surprising. Indeed, the
natural conjecture that the capacities $\kappa_i$, $1\le i\le n$,
$n\ge 3$ are semiadditive seems presently completely out of reach.
The reason is that one should develop real variables techniques
which replace the simple minded but extremely powerful complex
variable argument described above.

On the other hand, recall that if $T$ is a compactly supported
distribution with bounded Cauchy potential then
\begin{equation}\label{growthCauchy}
\begin{split}
|\langle T,\varphi_Q\rangle |&=\left |\left\langle T,\frac{1}{\pi
z}*\overline\partial\varphi_Q\right\rangle \right|=
\left|\left\langle \frac{1}{\pi z}*T, \overline\partial\varphi_Q
\right\rangle\right|\\*[7pt] &\leq \frac 1 \pi\left\|\frac 1{z}*T
\right\|_\infty \,\|\overline\partial\varphi_Q\|_{L^1(Q)}\leq \,
\frac 1 \pi\left\|\frac 1{z}*T \right\|_\infty \,l(Q),
\end{split}
\end{equation}
whenever $\varphi_Q$ is normalized by
$\|\overline\partial\varphi_Q\|_{L^1(Q)} \le l(Q). $ The preceding
argument extends to $\Rn$ for even dimensions $n = 2N$ as follows. A
standard Fourier transform computation shows that, for some constant
$c_n$ and each test function $\varphi$, one has
\begin{equation}\label{rep}
\varphi = c_n\, \sum_{j=1}^n \frac{x_j}{|x|^2}*
\partial_j (\triangle^{N-1})\varphi \equiv c_n\, \frac{x}{|x|^2}*
\nabla(\triangle^{N-1})\varphi.
\end{equation}
Let $T$ be a compactly supported real distribution with bounded
vector valued Riesz potential $x / |x|^2 * T$ and let $\varphi_Q$ a
function in $\cc^{n-1}(Q)$. Then
\begin{equation}\label{growthRiesz}
\begin{split}
|\langle T,\varphi_Q\rangle |&=\left |\left\langle T, c_n\,
\frac{x}{|x|^2}* \nabla(\triangle^{N-1})\varphi_Q \right\rangle
\right|= \left|\left\langle c_n\,\frac{x}{|x|^2}*T,
\nabla(\triangle^{N-1})\varphi_Q \right\rangle\right|\\*[7pt] &\leq
C\, \|\frac{x}{|x|^2}*T\|_\infty \,\|\nabla^{n-1}
\varphi_Q\|_{L^1(Q)}\leq  C \,\|\frac{x}{|x|^2}*T\|_\infty \,l(Q),
\end{split}
\end{equation}
provided $\varphi_Q$ is normalized by $\|\nabla^{n-1}
\varphi_Q\|_{L^1(Q)} \le l(Q).$ Here we are adopting the standard
convention of denoting by $\nabla^m \varphi$ the vector
$(\partial^s \varphi)_{|s|=m}$ of all $m$-th order partial
derivatives of $\varphi$ and by $ |\nabla^m \varphi|$ its Euclidean
norm.

For odd dimensions one has to require a stronger normalization
condition. The first remark is that \eqref{rep} can be rewritten as
\begin{equation}\label{rep2}
\varphi = c_n\, \frac{x}{|x|^2}* \nabla(-\Delta)^{(n-2)/2}\varphi,
\end{equation}
which makes sense for all dimensions.  Since
$$
(-\Delta)^{\frac{1}{2}}\varphi = d_n \, \sum_{j=1}^n R_j \partial_j
\varphi \,
$$
for some dimensional constant $d_n$, the $R_j$ being the Riesz
transforms (the Calder\'{o}n-Zygmund operators with Fourier multiplier
$\xi_j / |\xi|$), we have
$$
\nabla(-\Delta)^{(n-2)/2}\varphi = c_n\, \left(
\partial_j \left((\sum_{k=1}^n R_k \partial_k )^{n-2} \varphi \right) \right)_{j=1}^n.
$$
Each component of the vector in the right hand side above is a sum
of terms of the form $ T\,
\partial^s \varphi$, where $s$ is a multi-index of
length $n-1$ and $T$ is a product of $n-2$ Riesz transforms. Hence,
denoting by $\|\cdot\|_{H^1(\Rn)}$ the norm of the real Hardy space
$H^1(\Rn)$, we get
\begin{equation}\label{reproducing}
\|\nabla(-\Delta)^{(n-2)/2}\varphi\|_{L^1(\Rn)} \le C\,
\|\nabla^{n-1} \varphi\|_{H^1(\Rn)},
\end{equation}
where we have set
$$
\|\nabla^{n-1} \varphi\|_{H^1(\Rn)} = \sum_{|s|= n-1}
\|\partial^s \varphi\|_{H^1(\Rn)}.
$$
Recall that a function $f \in H^1(\Rn)$ if and only if $f \in
L^1(\Rn)$ and all its Riesz transforms are also in $L^1(\Rn)$. The
norm of $f$ in $H^1(\Rn)$ is defined as
$$
\|f\|_{H^1(\Rn)} = \|f\|_{L^1(\Rn)}+\sum_{j=1}^n
\|R_j(f)\|_{L^1(\Rn)}.
$$
A basic result is that the Riesz transforms send continuously
$H^1(\Rn)$ into itself, and this is what we used in
\eqref{reproducing}.

For even dimensions, as we have seen before, the Riesz transforms
disappear from the reproducing formula \eqref{rep2} and we get the
better estimate
$$
\|\nabla(-\Delta)^{(n-2)/2}\varphi\|_{L^1(\Rn)} \le C\,
\|\nabla^{n-1} \varphi\|_{L^1(\Rn)}.
$$
This accounts for the difference between even and odd dimensions.

  Let $T$ be a compactly
supported real distribution with bounded vector valued Riesz
potential $x / |x|^2 * T$ and let $\varphi_Q$ a function in
$C^{n-1}(Q)$. Therefore
\begin{equation}\label{growthRiesz1}
\begin{split}
|\langle T,\varphi_Q\rangle |&=\left |\left\langle T, c_n\,
\frac{x}{|x|^2}* \nabla(-\Delta)^{(n-2)/2} \varphi_Q \right\rangle
\right|= \left|\left\langle c_n\,\frac{x}{|x|^2}*T,
\nabla(-\Delta)^{(n-2)/2} \varphi_Q \right\rangle\right|\\*[7pt]
&\leq C\, \|\frac{x}{|x|^2}*T\|_\infty \,\|\nabla^{n-1}
\varphi_Q\|_{H^1(\Rn)}\leq  C \,\|\frac{x}{|x|^2}*T\|_\infty \,l(Q),
\end{split}
\end{equation}
provided $\varphi_Q$ is normalized by $\|\nabla^{n-1}
\varphi_Q\|_{H^1(\Rn)} \le l(Q).$

We say that a distribution $T$ has linear growth if
\begin{equation}\label{growthG}
G(T) = \sup_{\varphi_Q} \frac{|\langle T,\varphi_Q\rangle|}{l(Q)} <
\infty ,
\end{equation}
where the supremum is taken over all $\varphi_Q \in \cc^\infty_0(Q)$
satisfying the normalization inequalities
\begin{equation}\label{normalization}
  \|\partial^s\varphi_Q\|_{H^1(\Rn)} \leq l(Q), \quad
|s|= n-1.
\end{equation}

Notice that no distinction has been made between even or odd
dimensions in the preceding definition and that we have chosen the
stronger Hardy space normalization. This is due to the fact that,
since we will assume in our main result that the distributions we
deal with satisfy the linear growth condition, the stronger the
normalization we require the weaker the assumption we get.

The normalization in the $H^1$ norm is the right condition to
impose, as will become clear later on. For positive Radon measures
$\mu$ in $\Rn$ the preceding notion of linear growth is equivalent
to the usual one (see (\ref{lineargrowth}) below). In subsection 6.5
complete details on this fact are provided.

For a compact set $E$ in $\Rn$ we define $g(E)$ as the set of all
distributions supported on $E$ having linear growth with constant
$G(T)$ at most $1$. 




Our main result is a higher dimensional version of \eqref{mark}. For
a compact $E \subset \Rn$ set
$$
\Gamma(E)=\sup \left\{|\langle T ,1\rangle | \right\}
$$
where the supremum is taken over those real distributions $T$
supported on $E$ such that the vector field
$\displaystyle{\frac{x}{|x|^n}*T}$ is in the unit ball of
$L^\infty(\Rn, \Rn)$.  Hence $\Gamma(E)= \kappa(E)$ for $n=2$.
Finally, for $1\le k \le n$, set
$$
\Gamma_{\hat k}(E)=\sup \left\{|\langle T ,1\rangle| : T \in g(E)
\quad\text{and}\quad \left\|\frac{x_i}{|x|^{2}}*T\right\|_\infty
\leq 1,\, 1 \leq i \leq n,\, i \neq k\right \}.
$$
Thus we require the boundedness of $n-1$ components of the vector
valued potential~$x/|x|^{2}*T$ with Riesz kernel of homogeneity
$-1$.

The requirement of the growth condition in the preceding definition
is vital in obtaining the localization result
\eqref{scalarlocalization}. In subsection 6.4 we show that a growth
condition is necessary for a localization estimate in $L^\infty$.

\noindent Our extension of \eqref{mark} to $\Rn$ is the following.

\begin{teo*}
There exists a positive constant $C$ such that for each compact
set~$E\subset\Rn$ and $1 \le k \le n$
\begin{equation}\label{GammahatGamma}
C^{-1}\, \Gamma_{\hat k}(E) \le \Gamma(E) \le C\, \Gamma_{\hat
k}(E).
\end{equation}
\end{teo*}

The second inequality in \eqref{GammahatGamma} follows immediately
from the definitions of $\Gamma$ and~$\Gamma_{\hat k}$, because any
real distribution $T$ with bounded vector valued Riesz potential has
linear growth as shown in \eqref{growthRiesz} and
\eqref{growthRiesz1}.

The paper is organized as follows. In section 2 we present a sketch
of the proof of the Theorem. It becomes clear that the proof depends
on two facts: the close relationship between the quantities one
obtains after symmetrization of the kernels~$x/|x|^2$
and~$x_i/|x|^2$ and a localization $L^\infty$ estimate for the
scalar kernels $x_i/|x|^2$. In section~3 we deal with the
symmetrization issue and in section 4 with the localization
estimate. In section 5 we discuss three examples showing that
boundedness of $(n-1)$-scalar signed Riesz potentials $x_i/|x|^2 * T
$ does not imply a linear growth estimate on $T$. In section 6 we
present various additional results and examples. We show that
$\kappa_i(E)$ is finite for each compact $E$. We present
counter-examples to two natural inequalities. The first shows that
the obvious extension of the Theorem to the vector valued Riesz
kernels $x/|x|^{1+\alpha}$ and scalar kernels $x_i/|x|^{1+\alpha}$
of homogeneity $\alpha$, $0 < \alpha  < 1$, fails. The second
counter-example shows that the obvious extension of \eqref{mark} to
kernels of homogeneity $-d$, where $d$ is an integer greater than
$1$, also fails. Finally we point out that a growth condition is
necessary to have localization inequalities in $L^\infty$.

Our notation and terminology are standard. For instance,
$\cc^m_0(E),\; 0 \le m \le \infty,$ denotes the set of all functions
with compact support contained in the set $E$ and with continuous
partial derivatives up to order $m$. Cubes will always be supposed
to have sides parallel to the coordinate axis, $l(Q)$ is the side
length of the cube $Q$ and $|Q|=l(Q)^n$ its volume. A good reference
for the theory of the real Hardy space $H^1(\Rn)$ is
\cite[Chapters 3 and 4]{St2}.

We remind the reader that the convolution of two distributions $T$
and $S$ is well defined if $T$ has compact support. In this case the
action of $T*S$ on the test function~$\varphi$ is
$$
\langle T*S, \varphi\rangle = \langle T, S*\varphi\rangle,
$$
which makes sense because $S*\varphi$ is an infinitely
differentiable function on $\Rn$.

\section{Sketch of the proof of the Theorem}

As we remarked before, one only has to prove that
\begin{equation}\label{Gammabarret}
\Gamma_{\hat k}(E) \le C \,\Gamma(E).
\end{equation}
Clearly $\Gamma(E)$ is larger than or equal to
\begin{equation}\label{Gamma+}
\Gamma_+(E)= \sup \mu(E)
\end{equation}
where the supremum is taken over those positive measures $\mu$
supported on $E$ whose  vector valued Riesz potential
$x/|x|^{2}*\mu$ lies in the closed unit ball of
$L^\infty(\Rn,\Rn)$. Now, $\Gamma_+(E)$ is comparable to yet
another quantity $\Gamma_{\op}(E)$, that is, for some positive
constant $C$ one has
\begin{equation}\label{GammaopGamma+}
C^{-1}\, \Gamma_{\op}(E) \le \Gamma_+(E) \le C\,
\Gamma_{\op}(E),
\end{equation}
for each compact set $E \subset\Rn$ (see \cite{tolsa}). Before
giving the definition of $\Gamma_{\op}(E)$ we need to introduce
the Riesz transform with respect to an underlying positive Radon
measure $\mu$ satisfying the linear growth condition
\begin{equation}\label{lineargrowth}
\mu(B(x,r)) \le C\,r, \quad x \in \Rn, \quad r\geq 0.
\end{equation}
Given $\epsilon > 0$ we define the truncated Riesz transform at
level $\epsilon$ as
\begin{equation}\label{Rieszep}
R_{\epsilon}(f \,\mu)(x)=\int_{|y-x|>\epsilon} \frac{x-y}{|x-y|^2}
 f(y)\, d\mu(y), \quad x \in \Rn,
\end{equation}
for $f \in L^2(\mu)$. The growth condition on $\mu$ insures that
each $R_{\epsilon}$ is a bounded operator on $L^2(\mu)$ with
operator norm $ \|R_{\epsilon}\|_{L^2(\mu)}$ possibly depending on~$\epsilon$. We say that the Riesz transform is bounded on
$L^2(\mu)$ when
$$
\|R\|_{L^2(\mu)} = \sup_{\epsilon > 0} \|R_{\epsilon}\|_{L^2(\mu)}
< \infty,
$$
or, in other words, when the truncated Riesz
transforms are uniformly bounded on~$L^2(\mu)$. Call $L(E)$ the
set of positive Radon measures supported on $E$ which satisfy
\eqref{lineargrowth} with $C=1$ . One defines $\Gamma_{\op}(E)$ by
$$
\Gamma_{\op}(E) = \sup \{\mu(E): \mu \in L(E) \quad\text{and}\quad
\|R\|_{L^2(\mu)} \le 1 \}.
$$
From the first inequality in \eqref{GammaopGamma+} we get that, for some constant~$C$ and all compact sets~$E$,
$$
\Gamma_{\op}(E) \le C \, \Gamma(E).
$$
We remind the reader that the first inequality in
\eqref{GammaopGamma+} depends on a simple but ingenious duality
argument due to Davie and Oksendal (see \cite[p.139]{do},
 \cite[Theorem 23, p.107]{llibrechrist} and \cite[Lemma 4.2]{verdera}).
 To prove \eqref{Gammabarret} we have to
estimate $\Gamma_{\hat k}(E)$ by a constant times
$\Gamma_{\op}(E)$. The natural way to perform that is to introduce
the quantity $\Gamma_{\hat k,\op}(E)$ and try the two estimates
\begin{equation}\label{scalarGammaopGamma}
\Gamma_{\hat k}(E) \le C \,\Gamma_{\hat k,\op}(E)
\end{equation}
and
\begin{equation}\label{scalarGammaopGammaop}
\Gamma_{\hat k,\op}(E) \le C \,\Gamma_{\op}(E).
\end{equation}

We define the truncated scalar Riesz transform $R^i_\ep(f
\,\mu)(x)$ associated with the $i$-th coordinate as in~\eqref{Rieszep}
 with the vector valued Riesz kernel replaced by the scalar Riesz kernel
$\displaystyle{\frac{x_i-y_i}{|x-y|^2}}$. We also set
$$
\|R^i\|_{L^2(\mu)} = \sup_{\epsilon > 0}
\|R^i_\ep\|_{L^2(\mu)} ,
$$
and
$$
\Gamma_{\hat k, \op}(E) = \sup \{\mu(E): \mu \in L(E)
\quad\text{and}\quad \|R^i\|_{L^2(\mu)} \le 1,\, 1 \le i\le n,\,
i\neq k \}.
$$

One proves \eqref{scalarGammaopGammaop} by checking that
symmetrization of a scalar Riesz kernel is controlled by the
symmetrization of the scalar Riesz kernels associated with all other
variables. This result was known to Stephen Semmes many years ago
\cite {S}. Here the fact that we are dealing with kernels of
homogeneity $-1$ plays a key role, because, as it is well-known,
they enjoy a special positivity property which is missing in
general. See section 3 for complete details. For other
homogeneities, either the corresponding statements are false or open
(see section 6).

The proof of \eqref{scalarGammaopGamma} depends on Tolsa's
approach to the proof of \eqref{tolsa}, which extends without any
significant change to the higher dimensional setting to give
$$
\Gamma(E) \le C\, \Gamma_+(E).
$$

The main technical point missing in our setting is a localization
result for scalar Riesz potentials. This turns out to be a
delicate issue, which we deal with in section~4. Specifically, we
prove that there exists a positive constant~$C$ such that, for
each compactly supported distribution~$T$ and for each coordinate
$i$, we have
\begin{equation}\label{scalarlocalization}
\left\|\frac{x_i}{|x|^{2}} * \varphi_Q T \right\|_\infty  \leq
C\left(\left\|\frac{x_i}{|x|^{2}} * T \right\|_\infty+G(T)\right)
\end{equation}
for each cube $Q$ and each  $\varphi_Q \in{\mathcal C}^\infty_ 0(Q)$
satisfying $\|\partial^{s} \varphi_Q\|_\infty \le l(Q)^{-|s|}$,
$0\leq|s|\leq n-1$.

This improves significantly a previous localization result in
\cite{mpv}, which, in particular, yields
\begin{equation}\label{vectorlocalization}
\left\|\frac{x}{|x|^{2}} * \varphi_Q T\right \|_\infty  \leq C
\left\|\frac{x}{|x|^{2}} * T \right\|_\infty ,
\end{equation}
for $\varphi_Q$ as above. Inequality \eqref{scalarlocalization}
implies \eqref{vectorlocalization} because boudedness of the vector
valued potential $x/|x|^2*T$ provides a growth condition on~$T$.
Indeed one has (see Lemma~3.2 in \cite{laura1} or
\eqref{growthRiesz} and \eqref{growthRiesz1})
$$
G(T)\leq C\left\|\frac{x}{|x|^2}*T\right\|_{\infty}.
$$

Once \eqref{scalarlocalization} is at our disposition Tolsa's
machinery applies straightforwardly as was already explained in
\cite [Section 2.2]{mpv}. However we will again describe the main
steps in the proof of inequality \eqref{scalarGammaopGamma} at the
end of section~4.

\section{Proof of \boldmath$\Gamma_{\hat k, \op}(E)\leq C \,\Gamma_{\op}(E)$}\label{sectionpermu}

The symmetrization process for the Cauchy kernel introduced in
\cite{me} has been succesfully applied to many problems of analytic
capacity and $L^2$ boundedness of the Cauchy integral operator (see
\cite{mv}, \cite{mmv} and the book \cite{pa}, for example) and also
to problems concerning the capacities, $\gamma_\alpha$,
$0<\alpha<1$, (which are related to the vector valued Riesz kernels
$x/|x|^{1+\alpha}$) and the $L^2$ boundedness of the $\alpha$-Riesz
transforms (see \cite{laura1}, \cite{mpv}, \cite{laura3} and
\cite{pv}). Given $3$ distinct points in the plane, $z_1$, $z_2$ and
$z_3$, one finds out, by an elementary computation that
\begin{equation}\label{curvatura}
c(z_1,z_2,z_3)^2=\sum_{\sigma}\frac
1{(z_{\sigma(1)}-z_{\sigma(3)})\overline{(z_{\sigma(2)}-z_{\sigma(3)})}}
\end{equation}
where the sum is taken over the permutations of the set
$\{1,2,3\}$ and $c(z_1,z_2,z_3)$ is {\em Menger curvature}, that
is, the inverse of the radius of the circle through $z_1$, $z_2$ and~$z_3$. In particular \eqref{curvatura} shows that the sum on the
right hand side is a non-negative quantity.

In $\Rn$ and for $1\leq i\leq n$ the quantity
\begin{equation}
 \label{permui}
\sum_{\sigma}\frac{x^i_{\sigma(2)}-x^i_{\sigma(1)}}{|x_{\sigma(2)}-x_{\sigma(1)}|^2}\frac{x^i_{\sigma(3)}-x^i_{\sigma(1)}}{|x_{\sigma(3)}-x_{\sigma(1)}|^2}
\end{equation}
where the sum is taken over the permutations of the set
$\{1,2,3\}$, is the obvious analogue of the right hand side of~\eqref{curvatura} for the $i$-th coordinate of the Riesz kernel~$x/|x|^2$. Notice that \eqref{permui} is exactly
$$
2\, p_i(x_1,x_2,x_3),
$$
where $p_i(x_1,x_2,x_3)$ is defined as the
sum in~\eqref{permui}  taken only on the three permutations~$(1,2,3)$, $(3,1,2)$ and $(2,1,3)$.

In Lemma \ref{permutacions}, we will show that, given three
distinct points $x_1,x_2,x_3\in\Rn$, the quantity
$p_i(x_1,x_2,x_3)$, $1\leq i\leq n$, is also non-negative. We
will use this remarkable fact to study the $L^2$ boundedness of
the operators associated with the scalar Riesz kernels~$x_i/|x|^2$.

The relationship between the quantity $p_i(x_1,x_2,x_3)$, $1\leq
i\leq n$, and the $L^2$~estimates of the operator with
kernel~$x_i/|x|^2$ is as follows. Take a positive finite Radon
measure~$\mu$ in $\Rn$ with linear growth. Given $\ep>0$ consider
the truncated scalar Riesz transform $R^i_\ep(\mu)$ of $\mu$
associated with the kernel $x_i/|x|^2$, as in section 2. Then we
have (see in~\cite{mv} the argument for the Cauchy integral
operator)
\begin{equation}\label{l2perm}
\left|\int|R_{\ep}^i(\mu)(x)|^2\,d\mu(x)-\frac 1
3p_{i,\ep}(\mu)\right|\leq C\|\mu\|,
\end{equation}
$C$ being a positive constant depending only on $n$, and
$$
p_{i,\ep}(\mu)=\underset{S_\ep}{\iiint}p_i(x,y,z)\,d\mu(x)\,d\mu(y)\,d\mu(z),
$$
with
$$
S_\ep=\{(x,y,z):|x-y|>\ep,\, |x-z|>\ep \text{ and }|y-z|>\ep\}.
$$

\begin{lemma}\label{permutacions}
For $1\leq i\leq n$, and any three distinct points
$x_1,x_2,x_3\in\Rn$ we have
$$
p_i(x_1,x_2,x_3)\geq 0.
$$

Moreover,
\begin{enumerate}
\item If $p_i(x_1,x_2,x_3)=0$ for $n-1$ values of  $i\in\{1,2,\dotsc,n\}$,   then  $x_1$, $x_2$, $x_3$ are aligned.
\item  If the three points $x_1$, $x_2$, $x_3$ are aligned, then $p_i(x_1,x_2,x_3)=0$ for $1\leq i\leq n$.
\end{enumerate}
\end{lemma}

\begin{proof}
Write $a=x_2-x_1$ and $b=x_3-x_2$. Then
\begin{equation*}
\begin{split}
p_i(x_1,x_2,x_3)&=\frac{a_i(a_i+b_i)|b|^2-b_ia_i|a+b|^2+b_i(a_i+b_i)|a|^2}{|a|^2|b|^2|a+b|^2}\\*[7pt]
&=\frac{a_ib_i\left(-2\sum_{j=1}^na_jb_j\right)+\sum_{j=1}^na_i^2b_j^2+b_i^2a_j^2}{|a|^2|b|^2|a+b|^2}\\*[7pt]
&=\frac{\sum_{j=1}^n(a_ib_j-b_ia_j)^2}{|a|^2|b|^2|a+b|^2}=\frac{\sum_{j
:  j\neq i}(a_ib_j-b_ia_j)^2}{|a|^2|b|^2|a+b|^2}\geq 0.
\end{split}
\end{equation*}

Therefore, given three pairwise distinct points $x_1$, $x_2$, $x_3$,
the permutations $p_i(x_1,x_2,x_3)$ vanish if and only if
$a_ib_j=b_ia_j$ for all $1\leq j \leq n$.

Without loss of generality, assume that $p_i(x_1,x_2,x_3)=0$ for
$1\leq i\leq n-1$. Then the following $n(n-1)/2$ conditions hold
$$
a_ib_j=a_jb_i\quad 1\leq i\leq n-1, \quad i+1\leq j\leq n.
$$

These conditions imply that $a=\lambda b$, for some
$\lambda\in\mathbb{R}$, which means the three points $x_1$, $x_2$, $x_3$
lie on the same line.

Assume now that the three points are aligned. Without loss of
generality set $x_1=0$, $x_2=y$ and $x_3=\lambda y$ for some
$\lambda>0$, and $y\in\Rn$. Then for $i,j\in \{1,2,\dotsc, n\}$, we
have
$$
a_ib_j=y_i(\lambda-1)y_j=(\lambda-1)y_iy_j=b_ia_j,
$$
hence $p_i(x_1,x_2,x_3)=0$ for $1\leq i\leq n$.
\end{proof}

If we are in the plane, then Menger curvature can be written as
$$
c(x_1,x_2,x_3)=\frac{4A}{|x_1-x_2||x_1-x_3||x_3-x_2|},
$$
where $A$ denotes the area of the triangle determined by the
points $x_1$, $x_2$, $x_3$. A consequence of Lemma~\ref{permutacions}
and its proof is the following.

\begin{co}\label{permutacions2}
Given three different points $x_1,x_2,x_3\in\mathbb{R}^2,$ we
have
$$
p_1(x_1,x_2,x_3)=p_2(x_1,x_2,x_3)= \frac 1 4 c(x_1,x_2,x_3)^2.
$$
Hence, the quantities $p_1(x_1,x_2,x_3)$ and $p_2(x_1,x_2,x_3)$
are non-negative, and vanish if and only if  $x_1$, $x_2$, $x_3$
are aligned.
\end{co}

In the plane the singular Cauchy transform $C$ with respect to the
underlying measure $\mu$ may be written as
$C(f\mu)=R^1(f\mu)-iR^2(f\mu)$. By Corollary~\ref{permutacions2} and
the $T(1)$-Theorem , we see that $C$ is bounded on~$L^2(\mu)$ if and
only if one of its real components, no matter which one, is bounded
on $L^2(\mu)$. We state this, for emphasis, as a corollary.

\begin{co}\label{l2bound}
If $\mu$  is a compactly supported positive measure in the plane
having linear growth, the Cauchy transform of $\mu$ is bounded on
$L^2(\mu)$ if and only if $R^i$ is bounded on $L^2(\mu)$ for one
$i\in\{ 1,2\}$.
\end{co}

For a positive measure $\mu$ with linear growth we have, by \eqref{l2perm},
\begin{equation*}
\begin{split}
\|R_\ep(\mu)\|_{L^2(\mu)}^2&=\sum_{j=1}^n\int|R^j_{\ep}(\mu)(x)|^2\,d\mu(x)\\*[7pt]
&=\frac1{3}\sum_{j=1}^n\underset{S_{\ep}}{\iiint}p_j(x,y,z)\,d\mu(x)\,d\mu(y)\,d\mu(z)+
O(\|\mu\|)\\*[7pt]
&\leq \frac2{3}\sum_{\begin{subarray}{l}j=1\\j\neq
i\end{subarray}}^{n}\underset{S_{\ep}}{\iiint}p_j(x,y,z)\,d\mu(x)\,d\mu(y)\,d\mu(z)+
O(\|\mu\|),
\end{split}
\end{equation*}
where the last inequality follows easily from the formula
$$
p_i(x_1,x_2,x_3)=\frac{\sum_{j\neq i}\left((x_2^i-x_1^i)(x_3^j-x_2^j)-(x_2^j-x_1^j)(x_3^i-x_2^i)\right)^2}{|x_2-x_1|^2|x_3-x_2|^2|x_3-x_1|^2},\quad 1\leq i\leq n.
$$
The above estimate can be localized replacing $\mu$ by $\chi_B\mu$
for each ball~$B$. Therefore, appealing to the $T(1)$-Theorem for
non necessarily doubling measures \cite{ntv1}, if $n-1$ components
$R^j$ are bounded on $L^2(\mu)$ (no matter which $n-1$ components),
then the whole vector valued operator~$R$ is bounded on $L^2(\mu)$.

\begin{teo}\label{vector}
Let $\mu$ be a non-negative measure with compact support in $\Rn$
and linear growth. Then the vector valued Riesz operator~$R$
associated with the kernel $x/|x|^2$ is bounded on $L^2(\mu)$
provided any set of $n-1$ components $R^j$ of $R$ are bounded on
$L^2(\mu)$.
\end{teo}

The inequality \eqref{scalarGammaopGammaop} is an immediate
consequence of Theorem \ref{vector}.

\section{Proof of \boldmath$\Gamma_{\hat k}(E)\leq C\;\Gamma_{\hat k,\op}(E)$}

The proof of the inequality $\Gamma_{\hat k}(E)\leq C\, \Gamma_{\hat
k,\op}(E)$ is based in two ingredients, the localization of scalar
Riesz potentials and the exterior regularity of $\Gamma_{\hat k}$,
which we discuss below.

\subsection{Localization of scalar Riesz potentials}\label{local}

When analyzing the argument for the proof of \eqref{tolsa} (see
Theorem 1.1 in \cite{semiad}) one realizes that one of the technical
tools used is the fact that the Cauchy kernel~$1/z$ localizes in the
uniform norm. By this we mean that if $T$ is a compactly supported
distribution such that $1/z*T$ is a bounded measurable function,
then $1/z*(\varphi \, T)$  is also bounded measurable for each
compactly supported ${\mathcal C}^1$ function~$\varphi$. This is an old
result, which is simple to prove because $1/z$ is related to the
differential operator~$\overline{\partial}$ (see~\cite[Chapter V]{garnett}). The same localization result can be proved easily in any
dimension for the kernel~$x/|x|^n$, which is, modulo a
multiplicative constant, the gradient of the fundamental solution of
the Laplacian. Again the proof is reasonably straightforward because
the kernel is related to a differential operator (see~\cite{paramonov} and~\cite{verderacm}).

In \cite[Lemma 3.1]{mpv} we  were concerned with the localization of
the vector valued $\alpha$-Riesz kernel $x/|x|^{1+\alpha}$,
$0<\alpha<n$. For general values of $\alpha$  there is no
differential operator in the background and consequently the
corresponding localization result becomes far from obvious (see
Lemma 3.1 in \cite{mpv}).

\vspace*{7pt}

We now state the new localization lemma we need.

\begin{lemma}\label{localization1}
Let $T$ be a compactly supported distribution in $\Rn$, with linear growth,
such that $(x_i / |x|^2) *T$ is in $L^\infty(\Rn)$ for some $i$, $1\leq i \leq n$.
Let $Q$ be a cube and assume that $\varphi_Q \in \cc^\infty_0(Q)$
satisfies $\|\partial^{s} \varphi_Q\|_\infty \le l(Q)^{-|s|}$,
$0\leq|s|\leq n-1$.
 Then $(x_i / |x|^2) * \varphi_Q T$ is in $L^\infty(\Rn)$ and
$$
\left\|\frac{x_i}{|x|^2}*\varphi_Q T\right\|_\infty\leq C\left(\left\| \frac{x_i}{|x|^2}*T\right\|_\infty+G(T)\right),
$$
for some positive constant $C=C(n)$ depending only on $n$.
\end{lemma}

With analogous techniques and replacing $G(T)$ by $G_\alpha(T)$
(see section 6 \label{contraexemplealfa} for a definition) one can
prove that the above lemma also holds in $\Rn$ for the scalar
$\alpha$-Riesz potentials
$$
\frac{x_i}{|x|^{1+\alpha}}*T,\quad 0<\alpha<n,\;\alpha\in\Z.
$$

For the proof of Lemma \ref{localization1} we need the following.

\begin{lemma}\label{prelocalization}
 Let $T$ be a compactly supported distribution in $\Rn$ with linear growth and assume that
 $Q$ is a cube and $\varphi_Q \in \cc^\infty_0(Q)$
satisfies $\|\partial^{s} \varphi_Q\|_{\infty} \le l(Q)^{-|s|}$, $0\le |s|\le
n-1$. Then, for each coordinate $i$, the distribution $(x_i / |x|^2)
* \varphi_Q T$ is a locally integrable function and there exists a
point $x_0 \in \frac{1}{4}Q$ such that
$$
\left|\left(\frac{x_i}{|x|^2}*\varphi_QT\right)(x_0)\right|\leq C \, G(T),
$$
where  $C=C(n)$ is a positive constant depending only on $n$.
\end{lemma}

\begin{remark}
Since the function $f=(x_i/|x|^2)*\varphi_QT$ is only locally
integrable, it may look strange to evaluate $f$ at a point. Indeed
we show that the mean of $f$ on $\displaystyle{\frac 1 4 Q}$ is
bounded by $C\,G(T)$ and then at many Lebesgue points of $f$ the
above stated inequality holds, doubling the constant if
necessary.
\end{remark}

\begin{proof}[Proof of Lemma \ref{prelocalization}]
Without loss of generality set $i=1$ and write
$k^1(x)=x_1/|x|^2$. Since $k^1 * \varphi_Q T$ is infinitely
differentiable off the closure of ${Q}$, we only need to show
that $k^1* \varphi_Q T$~is integrable on $2Q$. We will actually
prove a stronger statement, namely, that $k^1* \varphi_Q T$~is in
$L^p(2Q)$ for each $p$ in the interval $1\leq p< \frac{n}{n-1}$. Indeed, fix
any $q$ satisfying $n <q < \infty$ and call $p$ the dual
exponent, so that $1 < p< \frac n{n-1}$.  We need to estimate the action of
$k^1 * \varphi_Q T$ on functions $\psi \in \cc^\infty_0(2Q)$ in
terms of $\|\psi\|_q $. We clearly have
$$
\langle k^1 * \varphi_Q T, \psi\rangle = \langle T, \varphi_Q(k^1 * \psi)\rangle.
$$
We claim that, for an appropriate
dimensional constant $C $, the test function
\begin{equation}\label{testf}
\frac{\varphi_Q(k^1 * \psi)}{C \,l(Q)^{\frac{n}{p}-1} \|\psi\|_q}
\end{equation}
satisfies the normalization inequalities \eqref{normalization} in
the definition of $G(T)$.  Once this is proved, by the definition of $G(T)$ we get
\begin{equation*}\label{Lq}
|\langle k^1 * \varphi_Q T, \psi\rangle | \le C\,
l(Q)^{\frac{n}{p}}\|\psi\|_q \,G(T),
\end{equation*}
and so
\begin{equation*}\label{Lp}
\|k^1 * \varphi_Q T \|_{L^p(2Q)} \le C\,
l(Q)^{\frac{n}{p}}G(T).
\end{equation*}
Hence
\begin{equation*}
\begin{split}
\frac{1}{|\frac{1}{4}Q|}\int_{\frac{1}{4} Q} |(k^1 * \varphi_Q
T)(x)|\,dx &\le 4^n\frac{1}{|Q|}\int_Q |(k^1 * \varphi_Q
T)(x)|\,dx \\*[7pt]
& \le 4^n\left(\frac{1}{|Q|}\int_Q |(k^1
* \varphi_Q T)(x)|^p \,dx\right)^{\frac{1}{p}}\\*[7pt]
& \le C\,G(T),
\end{split}
\end{equation*}
which completes the proof of Lemma \ref{prelocalization}.

To prove the claim we need the following auxiliary remark. We let
$R_i$ stand for the Riesz transforms, that is, the Calder\'{o}n-Zygmund
operators with kernel $c_n\, x_i/|x|^{n+1}$ and multiplier $\xi_i
/|\xi|.$
\begin{sublema*}\label{sublema}
Assume that $f_Q$ is a test function supported on the square $Q$
satisfying
\begin{equation*}\label{gradientn-1}
\|\partial^s f_Q\|_{L^1(Q)}\le l(Q), \quad |s|=n-1,
\end{equation*}
and
\begin{equation}\label{growthcondition}
\|R_i(\partial^sf_Q)\|_{L^1(2Q)}\le C l(Q).
\end{equation}
Then $$\|R_i(\partial^s f_Q)\|_{L^1(\Rn)}\le Cl(Q)\;\;\;\mbox{ for
}|s|=n-1\;\;\mbox{ and }1\le i\le n .$$
\end{sublema*}
\begin{proof}
For any multi-index $s$ with $|s|=n-1$, integrating by parts to
take one derivative on the Riesz kernel we obtain
\begin{equation*}
\begin{split}
\|R_i(\partial^s f_Q)\|_{L^1((2Q)^c)}&=c_n
\,\int_{(2Q)^c}|\int_Q\partial^s
f_Q(z)\frac{z_i-y_i}{|z-y|^{n+1}}\,dz|\,dy\\&\le C \|\nabla^{n-2}
f_Q\|_{L^1(Q)}\,l(Q)^{-1} \\& \le C\,\|\nabla^{n-2}
f_Q\|_{L^{n/n-1}(Q)},
\end{split}
\end{equation*}
where the last estimate follows from H\"{o}lder's inequality.

A well known result of Maz'ya (see \cite[1.1.4, p. 15]{mazya} and
\cite[1.2.2, p. 24]{mazya}) states that
\begin{equation}\label{Mazya}
\|\nabla^m f_Q\|_{\frac{n}{1+m}}\leq C\int|\nabla^{n-1}f_Q|, \quad
0\leq m \leq n-1 .
\end{equation}
Applying this for $m=n-2$ we get
\begin{equation*}
\begin{split}
\|R_i(\partial^s f_Q)\|_{L^1((2Q)^c)}& \le C\|\nabla^{n-2}
f_Q\|_{n/n-1} \\& \le C\|\nabla^{n-1} f_Q\|_1 \le C\,l(Q).
\end{split}
\end{equation*}
\end{proof}

By the Sublemma, to prove the claim we only have to show that for
$|s|=n-1$,
\begin{equation}\label{n}
\|\partial^s \left(\varphi_Q\,(k^1 * \psi)\right)\|_{L^1(Q)} \le
C\,l(Q)^{\frac{n}{p}}\|\psi\|_q .
\end{equation}
and
\begin{equation}\label{rieszineq}
\|R_i(\partial^s(\varphi_Q(k^1*\psi)))\|_{L^1(2Q)}\le
C\,l(Q)^{\frac{n}{p}}\|\psi\|_q,\quad 1\le i\le n.
\end{equation}

By H\"{o}lder's inequality and the fact that the Riesz transforms
preserve $L^q(\Rn)$, $1<q<\infty$, we get

\begin{equation*}
\begin{split}\|R_i(\partial^s(\varphi_Q(k^1*\psi)))\|_{L^1(2Q)}&\le C l(Q)^{\frac n p}\|R_i(\partial^s(\varphi_Q(k^1*\psi)))\|_{L^q(\Rn)}\\\\&\le Cl(Q)^{\frac n p}\|\partial^s(\varphi_Q(k^1*\psi))\|_{L^q(Q)}.
\end{split}
\end{equation*}
Hence \eqref{rieszineq} and \eqref{n} follow from

\begin{equation}\label{lq}\|\partial^s(\varphi_Q(k^1*\psi))\|_{L^q(Q)}\le C\|\psi\|_q.\end{equation}

By Leibnitz formula
\begin{equation*}\label{Leibnitz}
\begin{split}
\partial^s \left(\varphi_Q \,(k^1 * \psi)\right) & = \varphi_Q\,\partial^{s}(k^1*\psi) + \sum_{|r|=1}^{n-1}
c_{s,r}\,\partial^r\varphi_Q \;\partial^{s-r}(k^1*\psi)\\
&  \equiv A+B,
\end{split}
\end{equation*}
where the last identity is a definition of $A$ and $B$.

To estimate the $L^q$-norm of the function in $B$ we remark that, since $|s|=n-1$,
$$
|\partial^{s-r}k^1(x)| \le C\, |x|^{-(n-|r|)}, \quad 1 \le |r| \le
n-1 ,
$$
and then, by H\"{o}lder's inequality and $\|\partial^r
\varphi_Q\|_{\infty} \le l(Q)^{-|r|}$, $1\le r\le n-1$, we see
that
\begin{equation*}
\begin{split}
\|\partial^r\varphi_Q \, \partial^{s-r}(k^1*\psi)\|_{L^q(Q)}&\le
 C\,
 \|\partial^r\varphi_Q\|_{\infty}\left(\int_Q\left(\int_{2Q}\frac{|\psi(y)|}{|x-y|^{n-|r|}}dy\right)^qdx\right)^{1/q}\\&\le C\|\partial^r\varphi_Q\|_\infty\|\psi\|_q\left(\int_Q\left(\int_{2Q}\frac{dy}{|y-x|^{p(n-|r|)}} \right)^{\frac q p}dx\right)^{\frac 1 q}\\\\&\le
 C\,l(Q)^{-|r|}\|\psi\|_q l(Q)^{\frac n q}l(Q)^{\frac{n}{p}-n+|r|}\\\\&=C\|\psi\|_q,
\end{split}
\end{equation*}
for each $1\leq |r|\leq n-1$.
 We therefore  conclude that

$$\|B\|_q\le C\,\sum_{|r|= 1}^{n-1}\|\partial^r\varphi_Q\,\partial^{s-r}(k^1*\psi)\|_q\le C\,\|\psi\|_q.$$

 We turn now to the term $A$.  
We remark that, for $|s| = n-1$,
\begin{equation}\label{cz}
\partial^{s}k^1*\psi = c\,\psi + S(\psi),
\end{equation}
where $S$ is a smooth homogeneous convolution Calder\'{o}n-Zygmund
operator and $c$ a constant depending on~$s$. This can be seen by
computing the Fourier transform of $\partial^{s}k^1$ and then using
that each homogeneous polynomial can be decomposed in terms of
homogeneous harmonic polynomials of lower degrees (see \cite[3.1.2
p.~69]{St}). Since Calder\'{o}n-Zygmund operators preserve $L^q(\Rn)$,
$1 <q < \infty$, we get, using that $\|\varphi_Q\|_\infty\le C$,
$$
\|A\|_q  \le C\,\|\psi\|_q.
$$
This finishes the estimate of term $A$ and the proof of \eqref{lq}.
\qed
\renewcommand{\qedsymbol}{}
\end{proof}

\begin{proof}[Proof of Lemma \ref{localization1}]
Without loss of generality take $i=1$. Consider first a point $x \in
\Rn\setminus \frac 3 2 Q .$ Then $k^1(x-y)\,\varphi_Q(y)$ is in
$\cc^\infty_0(Q)$ as a function of $y.$ We claim that $c
\;l(Q)\,k^1(x-y)\,\varphi_Q(y)$ satisfies the normalization
condition \eqref{normalization} for some small constant $c$
depending only on $n$. Once the claim is proved we get
$$
|(k^1 * \varphi_Q T)(x)| = |\langle T, k^1(x-y)\,\varphi_Q(y)
\rangle| \le c^{-1}\,G(T).
$$
Straightforward estimates yield
$$
|\partial^s_y (k^1(x-y)\,\varphi_Q(y))| \le C\, l(Q)^{-n}, \quad
|s|=n-1,
$$
which shows that $\partial^s_y (k^1(x-y)\,\varphi_Q(y))$ is a
constant multiple of an atom, whence the claim.

We are then left with the case  $x\in \frac 3 2 Q$.  Since $k^1*T$
and $\varphi_Q$ are bounded functions, we can write
$$
 |(k^1*\varphi_QT)(x)|\leq|(k^1*\varphi_QT)(x)-\varphi_Q(x)(k^1*T)(x)|+\|\varphi_Q\|_\infty\|k^1*T\|_\infty.
$$

Let $\psi_Q\in{\mathcal C}_0^{\infty}(\Rn)$ be such that $\psi_Q\equiv
1$ in $2Q$, $\psi_Q\equiv 0$ in $(4Q)^c$ and
$\|\partial^s\psi_Q\|_\infty\leq C_s\,l(Q)^{-|s|}$, for each
multi-index~$s$. Then one is tempted to write
\begin{multline*}
 |(k^1*\varphi_QT)(x)-\varphi_Q(x)(k^1*T)(x)|\leq|\langle T,\psi_Q(y)(\varphi_Q(y)-\varphi_Q(x))k^1(x-y)\rangle|\\*[5pt]
 +\|\varphi_Q \|_{\infty}|\langle T,(1-\psi_Q(y))k^1(x-y)\rangle|.
\end{multline*}
The problem is that the first term in the right hand side above
does not make any sense because $T$ is acting on a function of $y$
which is not necessarily differentiable at the point $x$.  To
overcome this difficulty one needs to resort to a standard
regularization process. Take $\chi \in \cc^\infty(B(0,1))$ such
that $\int \chi(x)\,dx = 1$ and set $\chi_\ep(x)=
\ep^{-n}\,\chi(x/\ep)$. The plan is to estimate, uniformly on $x$
and $\epsilon$,
\begin{equation}\label{reg}
|(\chi_\ep*k^1*\varphi_QT)(x)-\varphi_Q(x)(\chi_\ep*k^1*T)(x)|.
\end{equation}
 Clearly \eqref{reg} tends, as $\ep$ tends to zero,
 to
$$
|(k^1*\varphi_QT)(x)-\varphi_Q(x)(k^1*T)(x)|,
$$
for almost all $x \in \Rn$, which allows the transfer of uniform
estimates. We now have
\begin{equation*}\label{dif}
\begin{split}
|(\chi_\ep*k^1*\varphi_QT)(&x)-\varphi_Q(x)(\chi_\ep*k^1*T)(x)|\\*[7pt]
&\le
|\langle T,\psi_Q(y)(\varphi_Q(y)-\varphi_Q(x))
(\chi_\ep*k^1)(x-y)\rangle|\\*[7pt]
&\quad+
\|\varphi_Q\|_{\infty}|\langle T,(1-\psi_Q(y))(\chi_\ep*k^1)(x-y)\rangle|\\*[7pt]
&=A_1+A_2,
\end{split}
\end{equation*}
where the last identity is the definition of $A_1$ and
$A_2$. To deal with term $A_1$ set
$$k^{1,x}_\ep(y)=(\chi_\ep*k^1)(x-y).$$ We claim that, for an
appropriate small dimensional constant $c$, the test function
$$f_Q(y)=c\;l(Q)\psi_Q(y)(\varphi_Q(y)-\varphi_Q(x))k_\ep^{1,x}(y),$$
satisfies the normalization inequalities \eqref{normalization} in
the definition of $G(T)$, with $\varphi_Q$ replaced by $f_Q$ and $Q$
by $4Q$. If this is the case, then
$$A_1\leq c^{-1} l(Q)^{-1}|\langle T,f_Q\rangle|\leq C\,G(T).$$

To prove the normalization inequalities \eqref{normalization} for
the function $f_Q$ it is enough, by the Sublemma, to show that the
following holds
\begin{equation}\label{growth1}
\|\partial^s f_Q\|_{L^1(4\,Q)}\le Cl(Q)
\end{equation}
\begin{equation}\label{growth2}
\|R_i(\partial^s f_Q)\|_{L^1(8\,Q)}\le Cl(Q),
\end{equation}
for $1\le i\le n$ and $|s|=n-1$.

For each $q>1$ let $p$ be its dual exponent. By H\"older's
inequality we have
\begin{equation*}
\|R_i(\partial^s f_Q)\|_{L^1(8Q)}\le Cl(Q)^{\frac n
p}\|R_i(\partial^s f_Q)\|_{L^q(\Rn)}\le Cl(Q)^{\frac n
p}\|\partial^s f_Q\|_{L^q(4Q)},
\end{equation*}
because the Riesz transforms preserve $L^q(\Rn)$. Therefore,
\eqref{growth1} and \eqref{growth2} are a consequence of
\begin{equation}\label{lq2}
\|\partial^s f_Q\|_{L^q(4Q)}\le Cl(Q)^{\frac n q-|s|},\;\;|s|=n-1.
\end{equation}

To prove \eqref{lq2} we first notice that the regularized kernel
$\chi_\ep*k^1$ satisfies the inequalities
\begin{equation}\label{regkernel}
|(\chi_\ep* \partial^s \,k^1)(x)| \le \frac{C}{|x|^{1+|s|}}, \quad
x \in \Rn\setminus \{0\}\quad \text{and}\quad 0 \le |s| < n-1,
\end{equation}
where $C$ is a dimensional constant, which, in particular, is
independent of $\epsilon$. This can be proved by standard
estimates which we omit. For $|s| = n-1$ the situation is a little
bit more complicated. By \eqref{cz} we have
\begin{equation*}
(\chi_\ep* \partial^s \,k^1)(x) = c\,\chi_\ep(x) + (\chi_\ep *
S)(x),
\end{equation*}
where $S$ is a smooth homogeneous convolution Calder\'{o}n-Zygmund
operator. As such, its kernel $H$ satisfies the usual growth
condition $|H(x)| \le C/ |x|^n$. From this is not difficult to
show that
\begin{equation}\label{regcz}
|(\chi_\ep * S)(x)| \le \frac{C}{|x|^{n}},\quad x \in \Rn\setminus
\{0\},
\end{equation}
for a dimensional constant $C$.

By Leibnitz formula, for $|s|=n-1$,
\begin{equation}\label{Leibnitz2}
\begin{split}
\partial^s \left(\psi_Q(\varphi_Q -\varphi_Q(x))k_{\ep}^{1,x}\right) & =
\psi_Q \,(\varphi_Q -\varphi_Q(x)) \partial^{s}\,
k_\ep^{1,x}\\*[7pt] &\quad +\sum_{|r|=1}^{n-1}c_{r,s}\,
\partial^r(\psi_Q(\varphi_Q
-\varphi_Q(x)))\;\partial^{s-r}\, k_\ep^{1,x},
\end{split}
\end{equation}
and so
\begin{equation*}
\begin{split}
\|\partial^s f_Q\|_{L^q(4Q)}&\leq Cl(Q)\left(\int_{4Q}|\psi_Q(y)
\,(\varphi_Q(y) -\varphi_Q(x))\, \partial^sk_{\ep}^{1,x}(y)|^q
\,dy\right)^{\frac 1 q}\\*[7pt] &\quad +Cl(Q)\sum_{|r|=1}^{n-1}\left(\int_{4Q}|\partial^r
\left(\psi_Q(\varphi_Q -\varphi_Q(x)\right)\,
\partial^{s-r}k_{\ep}^{1,x}(y) |^q\,dy\right)^{\frac 1 q}\\*[7pt]&
 =A_{11}+A_{12}.
\end{split}
\end{equation*}
Using \eqref{regkernel} one obtains
\begin{equation*}
\begin{split}
A_{12}&\le C l(Q)\sum_{|r|=1}^{n-1}\frac
1{l(Q)^{|r|}}\left(\int_{4Q}|(\partial^{s-r}k_\ep^{1,x})(y)|^q\,dy\right)^{\frac 1 q}
\\*[7pt] &\leq Cl(Q)^{\frac n q-|s|}.
\end{split}
\end{equation*}
To estimate $A_{11}$ we resort to \eqref{regcz}, which yields
\begin{equation*}
\begin{split}
A_{11}&=Cl(Q)\left(\int_{4Q} |\psi_Q(y)(\varphi_Q(y)
-\varphi_Q(x))\partial^sk_{\ep}^{1,x}(y) |^q \,dy\right)^{\frac 1 q}\\*[7pt] &\leq
Cl(Q)\|\partial\varphi_Q\|_\infty\left(\int_{4Q}
\frac{dy}{|y-x|^{q(n-1)}}\, dy\right)^{\frac 1 q}\\*[7pt] &\leq Cl(Q)^{\frac n q-|s|}
\end{split}
\end{equation*}

We now turn to $A_2$. By Lemma \ref{prelocalization}, there exists a
point $x_0\in Q$ such that $|(k^1*\psi_QT)(x_0)|\leq C\, G(T)$.
 Then
$$|(k^1*(1-\psi_Q)T)(x_0)|\leq C\,(\|k^1*T\|_\infty +G(T)).$$
The analogous inequality holds as well for the regularized
potentials appearing in $A_2$, uniformly in $\epsilon$, and
therefore
$$A_2\leq C\,|\langle T,(1-\psi_Q)(k_\ep^{1,x}-k_\ep^{1,x_0})\rangle|+C\,(\|k^1*T\|_\infty +G(T)).$$

To estimate $|\langle
T,(1-\psi_Q)(k_\ep^{1,x}-k_\ep^{1,x_0})\rangle|$, we decompose
$\Rn \setminus \{x\}$ into a union of rings $$N_j=\{z\in
\Rn:2^j\,l(Q)\leq|z-x|\leq 2^{j+1}\,l(Q)\},\quad j\in\mathbb{Z},$$
and consider functions $\varphi_j$ in ${\mathcal
C}^\infty_0(\Rn)$, with support contained in
$$N^*_j=\{z\in
\Rn:2^{j-1}\,l(Q)\leq|z-x|\leq 2^{j+2}\,l(Q)\},\quad
j\in\mathbb{Z},$$ such that $\|\partial^s\varphi_j\|_\infty\leq C
\,(2^j\,l(Q))^{-|s|}$, $|s| \geq 0$, and $\sum_j\varphi_j=1$ on
$\Rn\setminus\{x\}$. Since $x\in\frac 3 2 Q$  the smallest ring
$N^*_j$ that intersects $(2Q)^c$ is $N^*_{-3}$. Therefore  we have
\begin{equation*}
\begin{split}
 |\langle T,(1-\psi_Q)(k_\ep^{1,x}-k_\ep^{1,x_0})\rangle|
 &=\left|\left\langle T,\sum_{j\geq -3}\varphi_j(1-\psi_Q)(k_\ep^{1,x}-k_\ep^{1,x_0})\right\rangle\right|\\*[7pt]
 &\leq\left|\left\langle T,\sum_{j\in I}\varphi_{j}(1-\psi_Q)(k_\ep^{1,x}-k_\ep^{1,x_0})\right\rangle \right|\\*[7pt]
&\quad+\sum_{j\in J}|\langle T,\varphi_{j}(k_\ep^{1,x}-k_\ep^{1,x_0})\rangle|,
\end{split}
\end{equation*}
where $I$ denotes the set of indices $j\geq -3$ such that the
support of $\varphi_j$ intersects $4Q$  and $J$ the remaining
indices, namely those $j \geq -3 $ such that $\varphi_j$ vanishes
on $4Q$. Notice that the cardinality of $I$ is bounded by a
dimensional constant.

Set
$$g =C\,l(Q)\sum_{j\in I}\varphi_j(1-\psi_Q)\,(k_\ep^{1,x}-k_\ep^{1,x_0}),$$
and for $j\in J$
$$g_j=C\,2^{2j}\,l(Q)\,\varphi_j\,(k_\ep^{1,x}-k_\ep^{1,x_0}).$$
We show now that the test functions $g$ and $g_j$, $j\in J$,
satisfy the normalization inequalities \eqref{normalization} in
the definition of $G(T)$ for an appropriate choice of the (small)
constant $C$ . Once this is available, using the linear growth
condition of $T$ we obtain
\begin{equation*}
\begin{split}
 |\langle T,(1-\psi_Q)(k_\ep^{1,x}-k_\ep^{1,x_0})\rangle |&\leq C l(Q)^{-1}|\langle T,g\rangle|\\*[7pt]
 &\quad + C \sum_{j\in J} (2^{2j}l(Q))^{-1}|\langle T,g_j\rangle |\\*[7pt]
 &\leq C\,G(T) + C\sum_{j\geq -3}2^{-j}\,G(T)\leq C\,G(T),
\end{split}
\end{equation*}
which completes the proof of Lemma \ref{localization1}.

Checking the normalization inequalities for $g$ and $g_j$ is easy.
First notice that the support of $g$ is contained in a square
$\lambda\,Q$ for some dilation factor $\lambda$ depending only on
$n.$ On the other hand the support of $g_j$ is conained in $
2^{j+2}\,Q.$  By the Sublemma, we have to show that for $|s|=n-1$,
$1\le i \le n$,
\begin{equation}\label{lq3}
\|\partial^s g\|_{L^1(\lambda\, Q)}\le C
l(Q),\;\;\;\;\|R_i(\partial^s g)\|_{L^1(2\lambda\,Q)}\le Cl(Q)
\end{equation}
and for $1\le j\le n$,
\begin{equation}\label{lq4}
\|\partial^s g_j\|_{L^1(2^{j+2}\,Q)}\le C 2^j
l(Q),\;\;\;\;\|R_i(\partial^s g_j)\|_{L^1(2^{j+3}\,Q)}\le C 2^j
l(Q).
\end{equation}

As before, let $1 < q < \infty$ and call $p$ the dual exponent to
$q.$  Apply H\"older's inequality  and the fact that the Riesz
transforms preserve $L^q(\Rn)$ to obtain
\begin{equation*}
\|R_i(\partial^s g)\|_{L^1(2\lambda \,Q)}\le Cl(Q)^{\frac n
p}\|R_i(\partial^s g)\|_{L^q(\Rn)}\le Cl(Q)^{\frac n
p}\|\partial^s g\|_{L^q(\lambda\,Q)}
\end{equation*}
and
\begin{equation*}
\|R_i(\partial^s g_j)\|_{L^1(2^{j+3} Q)}\le
C\left(2^jl(Q)\right)^{\frac n p}\|R_i(\partial^s
g_j)\|_{L^q(\Rn)}\le C\left(2^jl(Q)\right)^{\frac n p}\|\partial^s
g_j\|_{L^q(2^{j+2}\,Q)}
\end{equation*}
hold. Therefore,  \eqref{lq3} and \eqref{lq4} follow from
\begin{equation}\label{final1}
\|\partial^s g\|_{L^q(\lambda\,Q)}\le C\,l(Q)^{\frac n q-|s|}
\end{equation}
and
\begin{equation}\label{final2}
\|\partial^s g_j\|_{L^q(2^{j+2}Q)}\le C\left(2^jl(Q)\right)^{\frac
n q-|s|}
\end{equation}
respectively.

To show \eqref{final1} we take $\partial^s$ in the definition of
$g$, apply Leibnitz's formula and estimate in the supremum norm each
term in the resulting sum . We get
$$
\|\partial^s g\|_\infty \le C\,l(Q) \sum_{|r|=0}^{n-1}
\frac{1}{l(Q)^{|r|}}\; \frac{1}{l(Q)^{1+|s|-|r|}} = C\,
\frac{1}{l(Q)^{|s|}},
$$
 which yields \eqref{final1} immediately.

For \eqref{final2}, applying a gradient estimate, we get
$$
|\partial^{s-r} k_\ep^{1,x}(y)- \partial^{s-r} k_\ep^{1,x_0}(y)|
\le C\, \frac{l(Q)}{(2^j\,l(Q))^{2+|s|-|r|}}, \quad y \in N^*_j,
\quad j \in J.
$$
Hence
$$
\|\partial^s g_j\|_\infty \le C\,2^{2j}\,l(Q) \sum_{|r|=0}^{n-1}
\frac{1}{(2^j\,l(Q))^{|r|}}\; \frac{l(Q)}{(2^j\,l(Q))^{2+|s|-|r|}} =
C\, \frac{1}{(2^j\,l(Q))^{|s|}},
$$
which yields \eqref{final2} readily.
\end{proof}

\subsection{A continuity property for the capacity \boldmath$\Gamma_{\hat k}$}

In this section we prove a continuity property for the capacity
$\Gamma_{\hat k}$, $1\leq k\leq n$, which will be used in the
proof of inequality \eqref{scalarGammaopGamma}. Although we state
the result only for the capacities $\Gamma_{\hat k}$, $1\leq k\leq
n$, Lemma \ref{extreg} below holds for the capacities $\kappa_i$,
$1\leq i\leq n$, defined in the Introduction, because the proof
does not use any growth condition on distributions with bounded
scalar Riesz potential.

\begin{lemma}\label{extreg}
 Let $\{E_j\}_j$ be a decreasing sequence of compact sets, with
 intersection the compact set $E\subset\Rn$. Then, for $1\leq k\leq n$,
$$\Gamma_{\hat k}(E)=\lim_{j\to\infty}\Gamma_{\hat k}(E_j).$$
\end{lemma}

\begin{proof}
Since, by definition,  the set function $\Gamma_{\hat
k}$ in non-decreasing
$$
\lim_{j\to\infty}\Gamma_{\hat k}(E_j)\geq\Gamma_{\hat k}(E),
$$
and the limit clearly exists. For each $j\geq 1$, let $T_j$ be a
distribution such that the potentials $x_i/|x|^2*T_j$ are in the
unit ball of $L^\infty(\Rn)$, $i \neq k$, and
$$
\Gamma_{\hat k}(E_j)-\frac 1 j < |\langle T_j,1\rangle | \leq\Gamma_{\hat k}(E_j).
$$
We want to show that for each test function $\varphi$,
\begin{equation}\label{test}
\langle T_j,\varphi\rangle \underset{j\to\infty}{\longrightarrow}\langle T,\varphi\rangle,
\end{equation}
for some distribution $T$ whose potentials $x_i/|x|^2*T$ are in
the unit ball of $L^\infty(\Rn)$ for $i\neq k$. If \eqref{test}
holds and $\varphi$ is a test function satisfying $\varphi\equiv
1$ in a neighbourhood of $E$, then
$$
\lim_{j\to\infty}\Gamma_{\hat k}(E_j)=\lim_{j\to\infty}|\langle T_j,1\rangle |=\lim_{j\to\infty}|\langle T_j,\varphi\rangle |= |\langle T,\varphi\rangle |\leq\Gamma_{\hat k}(E).
$$

To show \eqref{test}, fix $i \neq k$ and assume, without loss of
generality, that $i=1$. Set $k^1(x)=x_1/|x|^2$ and $f_j=k^1*T_j$.
Write a point $x \in \Rn$ as $x=(x_1,x_2)$, with $x_1\in\mathbb{R}$
and $x_2\in\mathbb{R}^{n-1}$. Finally notice that $c\,k^1 =
\partial_1 E$ where $E=log|x|$ and $c$ is a constant. Moreover, for each test function $\varphi$ one has

\begin{equation}\label{klaplacian}\varphi=c\Delta^{\frac n 2}\varphi*E,\end{equation}
for some constant $c$. For $n=2k$, identity (\ref{klaplacian}) says that $E$ is the fundamental solution of the
$k$-Laplacian in $\Rn$, and for $n=2k+1$, (\ref{klaplacian}) means that
\begin{equation}\label{oddcase}\varphi=c\Delta^{k+1}\varphi*\frac 1{|x|^{n-1}}*E.\end{equation}
We will only deal with the even case $n=2k$, since, by using the reproduction formula (\ref{oddcase}), the arguments for the odd case turn to be very similar.
Therefore, for each test
function $\varphi$,
$$
(T_j*\varphi)(x_1,x_2)=\int_{-\infty}^{x_1}\partial_1(T_j*\varphi)(t,x_2)\,dt=
c \int_{-\infty}^{x_1}\Delta^k(\varphi*f_j)(t,x_2)\,dt.
$$

Setting $\overline\varphi(x)=\varphi(-x)$ we get
\begin{equation}\label{firstcoordinate}
\langle T_j,\varphi\rangle =(T_j*\overline\varphi)(0,0)=c\int_{-\infty}^0\Delta^k(\overline\varphi*f_j)(t,0)\,dt.
\end{equation}
We remark, incidentally, that the above formula tells us how to
recover a distribution from one of its scalar Riesz potentials.

Passing to a subsequence, we can assume that $f_j\longrightarrow
f$ in the weak $*$ topology of $L^\infty(\Rn)$. But then
$(f_j*\Delta^k \varphi) (x) \longrightarrow (f*\Delta^k \varphi)
(x),\quad x \in \Rn$. This pointwise convergence is bounded
because $|(f_j*\Delta^k \varphi)
(x)|\leq\|\Delta^k\varphi\|_1\|f_j\|_{\infty} \le
\|\Delta^k\varphi\|_1$. Hence the dominated convergence theorem
yields
$$
\lim_{j\to\infty}\langle T_j,\varphi\rangle=c\lim_{j\to\infty}\int_{-\infty}^0\Delta^k(\overline\varphi*f_j)(t,0)\,dt =
c\int_{-\infty}^0\Delta^k(\overline\varphi*f)(t,0)\,dt.
$$

Define the distribution $T$ by
$$
\langle T,\varphi\rangle =c\int_{-\infty}^0\Delta^k(\overline\varphi*f)(t,0)\,dt.
$$

Now we want to show that $f=k^1*T$. For that we regularize $f_j$
and $T_j$. Take $\chi \in{\mathcal C}_0^{\infty}(B(0,1))$ with $\int
\chi(x)\,dx=1$ and set $\chi_\ep(x)=\ep^{-n}\chi(x/\ep)$. Then
we have, as $j \rightarrow \infty$,
$$
\left(\chi_\ep*k^1*T_j\right)(x)= \left(\chi_\ep*f_j\right)(x)
\longrightarrow \left(\chi_\ep*f\right)(x),\quad x \in \Rn,
$$
because $f_j$ converges to $f$ weak $*$ in $L^\infty(\Rn)$. On
the other hand, since $\chi_\ep *k_1 \in \cc^\infty(\Rn)$ and
$T_j$ tends to $T$ in the weak topology of distributions, with
controlled supports, we have
$$
\left(\chi_\ep*k^1*T_j\right)(x) \longrightarrow
\left(\chi_\ep*k^1*T\right)(x), \quad x \in \Rn.
$$
Hence
$$
\chi_\ep*k^1*T = \chi_\ep*f, \quad \ep > 0,
$$
and so, letting $\ep \rightarrow 0$, $k^1*T = f$.
\end{proof}

\subsection{End of the proof of the inequality \boldmath$\Gamma_{\hat k} \leq C\,\Gamma_{\hat k,\op}$}

We claim that the inequality in the title of this subsection  can be
proved by adapting the scheme of the proof of Theorems~1.1 in~\cite{semiad} and 7.1 in~\cite{semiad2}.  As Lemma \ref{extreg}
shows, the capacities $\Gamma_{\hat k}$, $1\leq k\leq n$, enjoy
the exterior regularity property. This is also true for the
capacities $\Gamma_{\hat k,+}$, $1\leq k\leq n$, defined by
$$
\Gamma_{\hat k,+}(E)=\sup\left\{\mu(E):\mu\in L(E),\,
\left\|\frac{x_j}{|x|^2}*\mu\right\|_\infty\leq 1,\,1\leq j \leq n,\, j\neq k\right\},
$$
just by the weak $\star$ compactness of the set of positive
measures with total variation not exceeding $1$.  
We can approximate a general compact set $E$ by sets which are
finite unions of cubes of the same side length in such a way that
the capacities $\Gamma_{\hat k}$ and $\Gamma_{\hat k,+}$ of the
approximating sets are as close as we wish to those of $E$.
As in \eqref{GammaopGamma+}, one has, using the Davie-Oksendal
Lemma for several operators \cite[Lemma 4.2]{mattilaparamonov},
\begin{equation}\label{comppositiu}
C^{-1}\, \Gamma_{\hat k,\op}(E) \le \Gamma_{\hat k,+}(E) \le C\,
\Gamma_{\hat k,\op}(E).
\end{equation}
Thus we can assume, without loss of generality,  that $E$ is a
finite union of cubes of the same size. This will allow to implement
an induction argument on the size of certain (n-dimensional)
rectangles. The first step involves rectangles of diameter
comparable to the side length of the cubes whose union is $E$.

The starting point of the general inductive step in the proof of
Tolsa's Theorem in~\cite{semiad} (and \cite{semiad2}) consists in
the construction of a positive Radon measure $\mu$ supported on a
compact set $F$ which approximates $E$ in an appropriate sense.  The
construction of $F$ and $\mu$ gives readily that $\Gamma_{\hat k}(E)
\le C\, \mu(F),$  and $ \Gamma_{\hat k,+}(F) \le C\,
\Gamma_{\hat k,+}(E)$, which tells us that $F$ is not too small
but also not too big.
 However, one cannot expect, in the context of \cite{semiad} and
\cite{semiad2}, the Cauchy singular integral to be
 bounded on~$L^2(\mu)$. In our case one cannot expect the
operators $R^j$ to be bounded on~$L^2(\mu)$, for $1\leq j\leq n$,
$j\neq k$. Here $R^j$ is the operator associated with the scalar
Riesz kernel $(x_j-y_j) /|x-y|^2$.One has to carefully look for a
compact subset $G$ of $F$ such that $\mu(F) \le C\,\mu(G)$, the
restriction $\mu_G$ of $\mu$ to $G$ has linear growth and the
operators $R^j$, $1\leq j\leq n$, $j\neq k$\,, are bounded on
$L^2(\mu_G)$ with dimensional constants. This completes the proof
because then
\begin{equation*}
\begin{split}
\Gamma_{\hat k}(E) &\le C\, \mu(F) \le C\, \mu(G) \le
C\,\Gamma_{\hat k,\op}(G)  \le C\,\Gamma_{\hat k,\op}(F) \\*[5pt]
& \le C\,\Gamma_{\hat k,+}(F) \le C\, \Gamma_{\hat k,+}(E) \le
C\,\Gamma_{\hat k,\op}(E) .
\end{split}
\end{equation*}

In \cite{semiad} and \cite{semiad2} the set $F$ is defined as the
union of a special family of cubes $\{Q_i\}_{i=1}^N$ that cover the
set $E$ and approximate $E$ at an appropriate intermediate scale.
One then sets
$$F=\bigcup_{i=1}^NQ_i.$$
This part of the proof extends without any obstruction to our case
because of  the positivity properties of the symmetrization of the
scalar Riesz kernels (see section~3). As in Lemma~7.2 in~\cite{semiad2}, just by how the approximating set~$F$ is
constructed, one gets $\Gamma_{\hat k,+}(F)\leq C\,\Gamma_{\hat
k,+}(E)$.  By the definition of $\Gamma_{\hat k}(E)$ it follows that
there exists a real distribution $T_0$ supported on $E$ such that
\begin{enumerate}
 \item ${\Gamma_{\hat k}(E) \le 2 |\langle T_0,1\rangle |.}$
 \item $T_0$ has linear growth and $G(T_0) \le 1$.
 \item  $\displaystyle{\|\frac{x_j}{|x|^2}* T_0\|_\infty\leq 1,}$ \quad $1\leq j\leq n$,\quad $j\neq  k$.
\end{enumerate}

Consider now functions $\varphi_i\in{\mathcal C}_0^{\infty}(2Q_i)$,
$0\leq\varphi_i\leq 1$, $\|\partial^s \varphi_i\|_\infty\leq
C\,l(Q_i)^{-|s|}$, $0 \leq |s|\leq  n-1$ and
$\sum_{i=1}^N\varphi_i=1$ on $\bigcup_iQ_i$. We define now
simultaneously the measure~$\mu$ and an auxiliary measure~$\nu$,
which should be viewed as a model for~$T_0$ adapted to the family of
cubes $\{Q_i\}_{i=1}^N$.  For each cube~$Q_i$ take a concentric
segment~$\Sigma_i$ of length a small fixed fraction of $\Gamma_{\hat
k}(E \cap Q_i)$ and set
$$
\mu=\sum_{i=1}^N\Hu_{|\Sigma_i}
$$
and
$$
\nu=\sum_{i=1}^N\frac{\langle T_0,\varphi_i\rangle }{\Hu(\Sigma_i)}\Hu_{|\Sigma_i}.
$$
We have $d\nu=bd\mu$, with
$\displaystyle{b=\frac{\langle \varphi_i,\nu_0\rangle}{\Hu(\Sigma_i)}}$ on
$\Sigma_i$. At this point we need to show that our function~$b$ is
bounded, to apply later a suitable $T(b)$ Theorem. To estimate
$\|b\|_{\infty}$ we use the localization inequalities
$$
\left\|\frac{x_j}{|x|^2}*\varphi_i T_0\right\|_\infty\leq C, \quad 1\leq j\leq n, \quad j\neq k, \quad 1\leq i\leq N .
$$
This was proved in Lemma \ref{localization1} of Section \ref{local}.
Since it is easily seen that $\varphi_i T_0$ has linear growth and
$G(\varphi_i T_0) \le C$, we obtain, by the definition of
$\Gamma_{\hat k}$,
\begin{equation}\label{capita}
 |\langle T_0 , \varphi_i\rangle |\leq C\,\Gamma_{\hat k}(2Q_i\cap E),\quad\text{for } 1\leq i\leq N.
\end{equation}
It is now easy to see why $\Gamma_{\hat k}(E) \le C \,\mu(F) $:
\begin{equation}
\begin{split}
\Gamma_{\hat k}(E) &\le 2\,|\langle T_0,1\rangle | = 2\left|\sum_{i=1}^N \langle T_0 ,
\varphi_i\rangle \right| \\*[7pt]
& \le C \sum_{i=1}^N  \Gamma_{\hat k}(2Q_i\cap E) =
C \, \mu(F).
\end{split}
\end{equation}

We do not insist in summarizing the intricate details, which can be
found in \cite{semiad} and \cite{semiad2}, of the definition of the
set $G$ and of the application of the $T(b)$ Theorem of \cite{ntv}.

\section{Counter-examples to the growth estimate}
As we explained in the introduction, if $T$ is a compactly
supported distribution such that $x/|x|^{2}*T$ is bounded,  then
$T$ satisfies the linear growth condition \eqref{creixement} (see
\eqref{growthRiesz} and \eqref{growthRiesz1}). This is no longer
true under the assumption that $n-1$ components of $x/|x|^{2}*T$
are bounded, as the following examples show.

\begin{prop}
There exist a compactly supported real Radon measure $\mu$  in $\mathbb{R}^n$, such that
for $1\le i\le n-1$, $x_i/|x|^2*\mu$ is in $L^\infty(\mathbb{R}^n)$ and $G(\mu) = \infty$.
\end{prop}

\begin{proof}
The idea of the proof is that there is no relation, in general,
between the derivatives of a function with respect to different
variables. The technical details of the proof differ according to
the parity of the dimension, so we deal separately with even and
odd dimensions. Indeed, we work in $\mathbb{R}^3$ and
$\mathbb{R}^4$, the general case being a straightforward extension
of these two.

\begin{enumerate}

 \item {\em The odd case:}

Set $x=(x_1,x_2,x_3)\in \mathbb{R}^3$ and let
$h(x)=f(x_1)f(x_2)g(x_3)$, where $f$ is the compactly supported
infinitely differentiable function defined by

\begin{equation}\label{f}
f(t)=
\begin{cases}
1&\text{if } t\in [0,1]\\*[5pt]
0&\text{if } t\in [-1,2]^c
\end{cases}.
\end{equation}

\noindent To define $g$, let $\psi$ be an infinitely
differentiable function supported on $[1/2,1]$, increasing for
$x\in[1/2,3/4]$, decreasing for $x\in[3/4,1]$ and such that
$\psi(3/4)=1$. Define $g$ on $I_j=[2^{-j-1},2^{-j}]$, $j\geq 0$,
by
\begin{equation}\label{g}
g(t)=\frac {\psi(2^j t)}{(j+1)^3},\,\,t\in I_j.
\end{equation}

\noindent Set $\mu=\partial_3^3h = f(x_1)f(x_2)g^{3)}(x_3),$ and
write $k^i(x)=x_i/|x|^{2}$, $1\le i\le 3$, so
 that, for $i=1,2$,

$$(\mu*k^i)(x)=(\partial_3^3h*k^i)(x)=(\partial_i\partial_3 h*\partial_3 k^3)(x).$$\\

\noindent We claim that  $\|k^i*\mu\|_\infty\le C, \;\; 1\le i\le
2$. Since for $m\geq 0$,
\begin{equation}\label{gk}
g^{m)}(t)=\frac{2^{mj}}{(j+1)^3}\psi^{m)}(2^j t),\,\,\,t\in I_j,
\end{equation}
we have,

\begin{equation*}
\begin{split}
|(\mu*k^i)(x)|&\le C\int_0^1\int_{-1}^2\int_{-1}^2\frac{ |f'(y_1)||f(y_2)||g'(y_3)|}{|y-x|^2}dy_1dy_2dy_3\\
&\le C\sum_{j}\frac{2^j}{(j+1)^3}\int_{-1}^2\int_{-1}^2\int_{I_j}\frac{dy_3dy_1dy_2}{|y-x|^2}\\
&\le C\sum_j\frac 1{(j+1)^2}\le C.
\end{split}
\end{equation*}
The next to the last inequality follows by decomposing the domain
of integration into "annuli" $\left(|y_1-x_1|^2+|y_2-x_2|^2
\right)^{1/2} \simeq 2^k\,2^{-j}, \;\;|y_3-x_3| \simeq 2^{-j},
\;\; 0 \le k$.

Hence $\|k^i*\mu\|_\infty\le C$ for $i=1,2$.

To see that the linear growth condition fails for the measure
$\mu$, take an interval $I^*_j \subset I_j $ such that, for some
fixed small positive number $\delta$, one has  $l(I^*_j) \geq
\delta \, l(I_j )$ and $g^{3)}(t) \geq \delta\,2^{3j}/(1+j)^3,\;\;
t \in I^*_j$. The existence of such $\delta$ and $I^*_j$ follows
readily from the definition of $g$ on $I_j.$ Take a non-negative
function $\phi \in {\cal C}^\infty_0(I^*_j)$ with $\phi(t) = 1 $
on $I^*_j/2$ (interval with the same center of $I^*_j$ and half
the length). Let $Q_j$ be the cube $(I^*_j)^3$, $j \ge 0$ and set
$\varphi_{Q_j}(x_1,x_2,x_3)=\phi(x_1)\phi(x_2)\phi(x_3)$. Then
$\varphi_{Q_j} \in {\cal C}_0^\infty(Q_j)$ and $\varphi_{Q_j}/ C$
 satisfies the normalization condition
\eqref{normalization} for some absolute big constant $C$. Then,
since $l(Q_j)=l(I^*_j)\approx l(I_j)=2^{-j}$, by \eqref{gk} for
$m=3$ we obtain,

\begin{equation*}
\begin{split}
\langle \mu,\varphi_{Q_j}\rangle\,&=\left(\int_{
I^*_j}\phi(t)dt\right)^2\,\int_{ I^*_j}\phi(t)g^{3)}(t)dt\approx
l(Q_j)^2 \frac{2^{3j}}{(j+1)^3} l(Q_j)=\frac{2^j}{(j+1)^3} l(Q_j).
\end{split}
\end{equation*}

\noindent Thus

$$
\frac{|\langle \mu,\varphi_{Q_j}\rangle|}{l(Q_j)}\underset{j\to\infty}{\longrightarrow}\infty,
$$
which implies $G(\mu)=\infty$.

\item {\em The even case:}

For $x=(x_1,x_2,x_3,x_4)\in \mathbb{R}^4$ let
$h(x)=f(x_1)f(x_2)f(x_3)g(x_4)$, where $f$ is the function defined
by \eqref{f} and $g$ is defined by
\begin{equation*}
g(t)= \psi(2^j t),\,\,t\in I_j,
\end{equation*}
that is, as in \eqref{g} except that the denominator $(j+1)^3$ is
not needed in this case.

Define $\mu=\Delta^2 h\,,$ the bilaplacian
of $h$. Then
\begin{equation*}
\begin{split}
\mu&=g(x_4)\left(f^{4)}(x_1)f(x_2)f(x_3)+f(x_1)f^{4)}(x_2)f(x_3)+f(x_1)f(x_2)f^{4)}(x_3)\right)\\*[7pt]&+2g(x_4)\left(f''(x_1)f''(x_2)f(x_3)+f''(x_1)f(x_2)f''(x_3)+f(x_1)f''(x_2)f''(x_3)\right)\\*[7pt]&+2g''(x_4)\left(f''(x_1)f(x_2)f(x_3)+f(x_1)f''(x_2)f(x_3)+f(x_1)f(x_2)f''(x_3)\right)\\*[7pt]&+
g^{4)}(x_4)f(x_1)f(x_2)f(x_3).
\end{split}
\end{equation*}

\noindent Write $k^i(x)=x_i/|x|^{2}$, $1\le i\le 4$. Notice that
$k^i(x)=c\,\partial_i E$, where $E$ is the fundamental solution of the bilaplacian and $c$ a constant. Then,
for $1\le i\le 3$,
$$
\|k^i*\mu\|_\infty=\|k^i*\Delta^2
h\|_\infty=\|c\,\partial_i(\Delta^2 h*E)\|_\infty=c\|\partial_i
h\|_\infty=c\|f\|_\infty^2 \|f'\|_\infty \|g\|_\infty\leq C.
$$
Although this is not necessary for the argument, notice that, by
\eqref{gk}, we have
$$
\|k^4*\mu\|_\infty=\|k^4*\Delta^2 h\|_\infty=\|c\,\partial_4(\Delta^2 h*E)
\|_\infty=c\|\partial_4 h\|_\infty=c\|f\|^3_\infty\|g'\|_\infty=\infty.
$$
Take an interval $I^*_j \subset I_j $ such that, for some fixed
small positive number $\delta$, one has  $l(I^*_j) \geq \delta \,
l(I_j )$ and $g^{4)}(t) \geq \delta\,2^{4j},\;\; t \in I^*_j$. The
existence of such $\delta$ and $I^*_j$ follows readily from the
definition of $g$ on $I_j.$ Take a non-negative function $\phi \in
{\cal C}^\infty_0(I^*_j)$ with $\phi(t) = 1 $ on $I^*_j/2$ (interval
with the same center of $I^*_j$ and half the length). Let $Q_j$ be
the cube $(I^*_j)^4$, $j \ge 0$ and set
$\varphi_{Q_j}(x_1,x_2,x_3,x_4)=\phi(x_1)\phi(x_2)\phi(x_3)\phi(x_4)$.
Then $\varphi_{Q_j} \in {\cal C}_0^\infty(Q_j)$ and $\varphi_{Q_j}/
C$
 satisfies the normalization condition
\eqref{normalization} for some absolute constant $C$. Then, since
$f''$ and $f^{4)}$ are zero on $I^*_j$ and $l(Q_j)= l(I^*_j)
\approx l(I_j)=2^{-j}$, by \eqref{gk} for $m=4$ we obtain,

\begin{equation*}
\begin{split}
\langle
\mu,\varphi_{Q_j}\rangle\,&=\left(\int_{I^*_j}\phi(t)f(t)dt\right)^3\,\int_{
I^*_j}\phi(t)g^{4)}(t)dt\\\\&\approx l(Q_j)^3\, 2^{4j}\,
l(Q_j)\approx 2^j \,l(Q_j).
\end{split}
\end{equation*}

\noindent Thus

$$
\frac{|\langle \mu,\varphi_{Q_j}\rangle|}{l(Q_j)}\underset{j\to\infty}{\longrightarrow}\infty,
$$
which implies $G(\mu)=\infty$.
\end{enumerate}
\end{proof}

On the plane, we do also have a counterexample in the setting of positive
measures, based on a completely different idea.

\begin{prop}
There exists a positive Radon measure $\mu$ such that
$x_1/|x|^2*\mu$ is in $L^\infty(\mathbb{R}^2)$ and $G(\mu) =
\infty$.
\end{prop}

\begin{proof}
Consider the function $f(t)=\log^+\dfrac{1}{|t|}$, $t\in
\mathbb{R}$. Then  $ f \in \mathit{BMO}(\mathbb{R})\setminus
L^{\infty}(\mathbb{R})$ and $f$ is supported on the interval
$[-1,1]$.  If $y >0$, then
\begin{equation*}
\begin{split}
\left(\frac{i}{\pi z}*f\right)(x,y)&=\frac 1\pi(k^2*f)(x,y)+\frac
i\pi(k^1*f)(x,y)\\*[7pt]
&=\frac{1}{\pi}\int_{\mathbb{R}}\frac{y}{(x-t)^2+y^2}f(t)\,dt+\frac{i}{\pi}\int_{\mathbb{R}}\frac{x-t}{(x-t)^2+y^2}f(t)\,dt\\*[7pt]
&=(P_y f)(x)+i(Q_y f)(x),
\end{split}
\end{equation*}
where $P_yf(x)$ and $Q_yf(x)$ are the Poisson transform and the
conjugate Poisson transform of $f$ respectively.

\vspace*{7pt}

\noindent Therefore, if $\displaystyle{Hf=\frac 1\pi \text{p.v.}
\frac 1x*f}$ is the Hilbert transform of $f$,
$$
(k^1*fdt)(x,y)=(Q_y f)(x)=P_y(Hf)(x).
$$
We claim that
\begin{equation}\label{hilbert}
H(f)\in
L^{\infty}(\mathbb{R}).
\end{equation}
If \eqref{hilbert} holds, then the positive measure
$\mu=f(t)\,dt$ satisfies
\begin{equation*}
|(k^1*\mu)(x,y)|=|P_y(Hf)(x)| \le \|Hf\|_\infty, \quad x \in
\mathbb{R}, \quad y > 0.
\end{equation*}
Since  $(k^1*\mu)(x,-y) = (k^1*\mu)(x,y)$, we get $k^1*\mu \in
L^\infty({\mathbb{R}}^2) $ and, on the other hand,  $\mu$ has not
linear growth, just because $f$ is unbounded.

To show \eqref{hilbert}, we first observe that integrating by
parts we have
$$
\text{p.v.}\int_{-1}^1\log\frac 1{|t|}\frac{dt}{x-t} =
\lim_{\epsilon \rightarrow 0} \int_{1 >|t|> \epsilon} \log
|x-t|\,\frac{dt}{t}.
$$
The function above is odd and so we can assume that $x$ is
positive. Making first the change of variables $\tau = -t$ and
then $u= t/x$ we get
\begin{equation*}
\begin{split}
\lim_{\epsilon \rightarrow 0} \int_{1 > |t|> \epsilon} \log
|x-t|\,\frac{dt}{t} & = - \lim_{\epsilon \rightarrow 0} \int_{1
> |t|> \epsilon} \log |x+t|\,\frac{dt}{t}\\*[4pt] &= \frac{1}{2}\,
\int_{-1}^1 \log \frac{|x-t|}{|x+t|}\,\frac{dt}{t}
\\*[4pt] &=
\frac{1}{2}\, \int_{-\frac{1}{x}}^{\frac{1}{x}} \log
\frac{|u-1|}{|u+1|}\,\frac{du}{u} .
\end{split}
\end{equation*}
Hence
$$
\left| \text{p.v.}\int_{-1}^1\log\frac 1{|t|}\frac{dt}{x-t}
\right|\le \frac{1}{2}\,\int_{-\infty}^{\infty} \left| \log
\frac{|u-1|}{|u+1|}\,\frac{1}{u} \right|\,du,
$$
which completes the proof because the last integral above is
finite.
\end{proof}

It is worth mentioning that we do not know whether there exists a
positive measure $\mu$ in $\mathbb{R}^n, \, n \geq3$, with the $n-1$
potentials $\mu * x_i/|x|^2$, $1 \le i \le n-1$, in
$L^\infty(\mathbb{R}^n)$,  but not having linear growth.

\section{Miscellaneous related results}

As we have seen in the previous sections, the fact that the Cauchy
kernel is complex is not as relevant as the fact that it is odd and
has homogeneity~$-1$. Indeed, in the plane, \eqref{mark} shows that
 one recovers the theory of analytic capacity by replacing the Cauchy kernel $1/z$ by any
  of the real kernels  $\operatorname{Re}(1/z)$ or $\operatorname{Im}(1/z)$. In $\Rn$, $n\ge 3$,
  the Theorem shows  that an analogue of \eqref{mark} holds in higher dimensions
  adding appropriate growth conditions on the admissible distributions.

A natural question is how one can extend this kind of results to the
higher dimensional real variable setting in which the kernel
$x/|x|^2$ is replaced by the vector valued Riesz kernel of
homogeneity $-\alpha$
$$
k_\alpha(x)=\frac x{|x|^{1+\alpha}},\quad x\in\Rn,\quad 0<\alpha<n,
$$
and the capacity associated with this kernel is defined by (see
\cite {laura1})
$$
\Gamma_\alpha(E)=\sup\left\{|\langle T,1\rangle|:\operatorname{spt}(T)\subset E,\,\left\|\frac{x}{|x|^{1+\alpha}}*T\right\|_\infty\leq 1\right\}.
$$
 The case $\alpha=n-1$, $n\geq 2$, is especially interesting,
because it gives Lipschitz harmonic capacity (see \eqref{LipCapAlt}).

Unfortunately, as we show in subsections 6.1 and 6.2 below, the most
obvious analogues of \eqref{mark} and the Theorem  fail in this
setting.

\subsection{Capacities associated with scalar \boldmath$\alpha$-Riesz potentials}
\label{counterexamplealfa}

Let $T$ be a compactly supported distribution in $\Rn$ and
$0<\alpha<n$.  As it was explained in the Introduction, the natural
notion of distribution $T$ of growth $\alpha$ should involve Hardy
spaces.  In our present case, one should replace the reproduction formula \eqref{rep2} by the following ones, depending on the nature of the parameter $\alpha$:

\begin{itemize}
\item $\alpha\in\Z$.   A standard Fourier transform computation shows that, for some constant $c_n$ and each test function $\varphi$, one has
$$\varphi=c_n\sum_{j=1}^n\frac{x_j}{|x|^{1+\alpha}}*\partial_j(\Delta^{(n-1-\alpha)/2})\varphi\equiv c_n\frac{x}{|x|^{1+\alpha}}*\nabla(\Delta^{(n-\alpha-1)/2})\varphi.$$
\item $\alpha\notin\Z$. A standard Fourier transform computation shows that, for some constant $d_n$ and each test function $\varphi$, one has
\begin{equation*}\begin{split}
\varphi=d_n\sum_{j=1}^n\frac{x_j}{|x|^{1+\alpha}}*\frac 1{|x|^{n-\{\alpha\}}}*\partial_j(\Delta^{(n-[\alpha]-1)/2})\varphi\\\equiv d_n\frac{x}{|x|^{1+\alpha}}*\frac 1{|x|^{n-\{\alpha\}}}*\nabla(\Delta^{(n-[\alpha]-1)/2})\varphi,\end{split}\end{equation*}
where $\alpha=[\alpha]+\{\alpha\}$, with $[\alpha]\in\Z$ and $\{\alpha\}\in (0,1)$.
\end{itemize}

Now we are able to define the notion of a compactly supported
distribution with growth $\alpha$, $0<\alpha<n$. We say that $T$ has
growth $\alpha$ provided
\begin{equation}\label{growthalfa}
G_\alpha(T) = \sup_{\varphi_Q} \frac{|\langle T,\varphi_Q\rangle|}{l(Q)^{\alpha}}
< \infty ,
\end{equation}
where the supremum is taken over all $\varphi_Q \in \cc^\infty_0(Q)$
satisfying the following normalization inequalities :

\begin{enumerate}

\item For $\alpha\in\Z$, we require \begin{equation}\label{integerh1}\|\partial^s\varphi_Q\|_{H^1(\Rn)}\le l(Q)^{\alpha},\;\;\;\;|s|=n-\alpha.\end{equation}


\item For $\alpha\notin\Z$, we  require \begin{equation}\label{alfah1}\|\partial^s\varphi_Q*\frac 1{|x|^{n-\{\alpha\}}}\|_{H^1(\Rn)}\le l(Q)^{\alpha},\;\;\;\;|s|=n-[\alpha].\end{equation}


\end{enumerate}

For positive Radon measures
$\mu$ in $\Rn$ the preceding notion of growth $\alpha$ is equivalent
to the usual one. In subsection 6.5
complete details on this fact are provided.

For a
compact set $E$ in $\Rn$ we define $g_{\alpha}(E)$ as the set of all
distributions $T$ supported on $E$ having growth $\alpha$ with
constant $G_\alpha(T)$ at most $1$.

For each coordinate $k$ set
$$
\Gamma_{\alpha,\,\hat k}(E)=\sup\{|\langle T,1\rangle|\},
$$
where the supremum is taken over those distributions $T\in
g_\alpha(E)$, such that the $j$-th component of the $\alpha$-Riesz
potential $x_j/|x|^{1+\alpha}*T$ is in the unit closed ball of~$L^\infty(\Rn)$, for $1 \le j \le n$, $j \neq k$.

The proof of Lemma 3.2 in \cite{laura1} tells us that if $k_\alpha *
T$ is in the unit ball $L^\infty(\Rn,\Rn)$, then the distribution
$T$ has $\alpha$-growth and $G_\alpha(T)\le C$. Hence
$\Gamma_\alpha(E)\leq C\,\Gamma_{\alpha,\, \hat k}(E)$. In this
section we prove the following
\begin{prop}\label{alphano}
Given $0<\alpha<1$, there exists a set $E\subset\Rn$ such that
$\Gamma_{\alpha}(E)=0$ and $\Gamma_{\alpha,\,\hat k}(E)>0$.
\end{prop}
Therefore
 $\Gamma_{\alpha}$ and $\Gamma_{\alpha,\, \hat k}$ are not comparable and thus the direct analogue of the Theorem
 fails in this setting.

We proceed now to symmetrize the scalar $\alpha$-Riesz kernels in
order to get a better understanding of the capacities
$\Gamma_{\alpha, \, \hat k}$ for $1\leq k\leq n$ and $0<\alpha<1$.

For $0<\alpha<n$ and $1\leq i\leq n$ the quantity
\begin{equation}\label{permuialfa}
\sum_{\sigma}\frac{x^i_{\sigma(2)}-x^i_{\sigma(1)}}{|x_{\sigma(2)}-x_{\sigma(1)}|^{1+\alpha}}\frac{x^i_{\sigma(3)}-x^i_{\sigma(1)}}{|x_{\sigma(3)}-x_{\sigma(1)}|^{1+\alpha}}
\end{equation}
where the sum is taken over the permutations of the set $\{1,2,3\}$,
is the analogue of the right hand side of~\eqref{permui} for the
$i$-th coordinate of the Riesz kernel $k_\alpha$. Notice that
\eqref{permuialfa} is exactly
$$
2\,p_{\alpha,\,i}(x_1,x_2,x_3),
$$
where $p_{\alpha,\,i}(x_1,x_2,x_3)$ is defined as the sum
in~\eqref{permuialfa} only taken on the three permutations~$(1,2,3),
(2,3,1)$ and $(3,1,2)$.

We will now show that given three distinct points $x_1, x_2,
x_3\in\Rn$, for $1\leq i\leq n$ and $0<\alpha\leq 1$, the quantity
$p_{\alpha,\,i}(x_1,x_2,x_3)$ is non-negative. We will use this to
study the $L^2$ boundedness of the scalar Riesz integral operator of
homogeneity $-\alpha$.

The relationship between the quantity $p_{\alpha,\,i}(x,y,z)$,
$0<\alpha\leq 1$, $1\leq i\leq n$, and the $L^2$ estimates of the
operator with kernel $k_\alpha^i=x_i/|x|^{1+\alpha}$ is as in
\eqref{l2perm}. That is, if $\mu$ is a positive finite Radon measure
in $\Rn$ with $\alpha$-growth, $\ep>0$ and we set
$$R_{\alpha,\, \ep}^i(\mu)(x)=\int_{|y-x|>\ep}k_\alpha^i(y-x)\,d\mu(y),
$$
then (see in \cite{mv} the argument for the Cauchy singular integral
operator)
\begin{equation}\label{mevescalar}
\left|\int|R_{\alpha,\, \ep}^i(\mu)(x)|^2\,d\mu(x)-\frac 1
3p_{\alpha,\,i,\,\ep}(\mu)\right|\leq C\|\mu\|,
\end{equation}
$C$ being a positive constant depending only on $n$ and $\alpha$,
and
$$
p_{\alpha,\,i,\,\ep}(\mu)=\underset{S_\ep}{\iiint}p_{\alpha,\,i}(x,y,z)\,d\mu(x)\,d\mu(y)\,d\mu(z),$$
with $$S_\ep=\{(x,y,z):|x-y|>\ep,|x-z|>\ep \text{ and
}|y-z|>\ep\}.$$

\vspace{.5cm}

\begin{lemma}\label{permutalfa}
 Let $0<\alpha<1$ and $x_1$, $x_2$, $x_3$ three different points in $\Rn$. For $1\leq i\leq n$ we have
\begin{equation}\label{projections}
\frac{(2-2^{\alpha})m^2}{L(x_1,x_2,x_3)^{2+2\alpha}}\leq
p_{\alpha,\,i}(x_1,x_2,x_3)\leq
\frac{3m^2}{L(x_1,x_2,x_3)^{2+2\alpha}},
\end{equation}
where
$m=\max(|x_2^i-x_1^i|,|x_3^i-x_2^i|,|x_3^i-x_1^i|)$ and
$L(x_1,x_2,x_3)$ is the length of the largest side of the triangle
determined by the three points $x_1$, $x_2$, $x_3$.

\vspace*{7pt}

\noindent Moreover, $p_{\alpha,\,i}(x_1,x_2,x_3)=0$ if and only if
the three points lie on a $(n-1)$\guio{hypersurface} perpendicular
to the $i$ axis, i.e. $x_1^i=x_2^i=x_3^i$.
\end{lemma}

\begin{proof} Without loss of generality fix $i=1$. Write $a=x_2-x_1$
and $b=x_3-x_2$, then $a+b=x_3-x_1$. A simple computation yields
\begin{equation}\label{alfa}
p_{\alpha,1}(x_1,x_2,x_3)=\frac{a_1^2|b|^{1+\alpha}+b_1^2|a|^{1+\alpha}+a_1b_1
\left(|b|^{1+\alpha}+|a|^{1+\alpha}-|a+b|^{1+\alpha}\right)}{|a|^{1+\alpha}|b|^{1+\alpha}|a+b|^{1+\alpha}},
\end{equation}
which makes the second inequality in~\eqref{projections} obvious. To
prove the first inequality in~\eqref{projections}, assume without
loss of generality, that $1=|a|\leq|b|\leq|a+b|$. Then
$$
p_{\alpha,1}(x_1,x_2,x_3)=\frac{1}{|b|^{1+\alpha}|a+b|^{1+\alpha}}\left(a_1^2|b|^{1+\alpha}+b_1^2+a_1b_1(1+|b|^{1+\alpha}-|a+b|^{1+\alpha})\right).
$$

We distinguish now two cases,
\begin{itemize}
 \item {\sl Case $a_1b_1\leq 0$.}
 Notice that since $|b|\leq|a+b|$,
 $$
 a_1b_1(1+|b|^{1+\alpha}-|a+b|^{1+\alpha})\geq a_1b_1.
 $$
 Then, since $|b|\geq 1$,
\begin{equation*}
\begin{split}
p_{\alpha,1}(x_1,x_2,x_3)&=\frac{1}{|b|^{1+\alpha}|a+b|^{1+\alpha}}\left(a_1^2|b|^{1+\alpha}\!+\!b_1^2\!+\!a_1b_1(1\!+\!|b|^{1+\alpha}\!-\!|a+b|^{1+\alpha})\right)\\*[7pt]
&\geq \frac{a_1^2|b|^{1+\alpha}+b_1^2+a_1b_1}{|b|^{1+\alpha}|a+b|^{1+\alpha}}\geq\frac{a_1^2+b_1^2+a_1b_1}{|b|^{1+\alpha}|a+b|^{1+\alpha}}\\*[7pt]
&=\frac 1 2\frac{(a_1+b_1)^2+a_1^2+b_1^2}{|b|^{1+\alpha}|a+b|^{1+\alpha}}.
\end{split}
\end{equation*}
\item {\sl Case $a_1b_1>0$.} Then $\max\{a_1^2,b_1^2,(a_1+b_1)^2\}=(a_1+b_1)^2$. Write $t=|b|\geq 1$ and
$$
f(t)=a_1^2t^{1+\alpha}+b_1^2+a_1b_1\left(1+t^{1+\alpha}-(1+t)^{1+\alpha}\right).
$$
By the triangle inequality,
$$
 p_{\alpha,1}(x_1,x_2,x_3)\geq \frac{f(t)}{|b|^{1+\alpha}|a+b|^{1+\alpha}}\geq\frac{\min_{t\geq 1}f(t)}{|b|^{1+\alpha}|a+b|^{1+\alpha}}.
$$

Our function $f$ has a minimum at the point
$t^*=\left(\left(\frac{a_1}{b_1}+1\right)^{1/\alpha}-1\right)^{-1}$.
\begin{enumerate}
\item If $a_1/b_1\geq 2^\alpha-1$, then $t^*\leq 1$. Therefore
\begin{equation*}
\begin{split}
p_{\alpha,1}(x_1,x_2,x_3)&\geq \frac{f(1)}{|b|^{1+\alpha}|a+b|^{1+\alpha}}\\*[7pt]
&=\frac{a_1^2+b_1^2+2a_1b_1(1-2^\alpha)}{|b|^{1+\alpha}|a+b|^{1+\alpha}}\\*[7pt]
&=(2^\alpha-1)\frac{(a_1-b_1)^2}{|b|^{1+\alpha}|a+b|^{1+\alpha}}+(2-2^{\alpha})\frac{a_1^2+b_1^2}{|b|^{1+\alpha}|a+b|^{1+\alpha}}\\*[7pt]
&\geq\frac{2-2^{\alpha}}{2}\frac{(a_1+b_1)^2}{|b|^{1+\alpha}|a+b|^{1+\alpha}}.
\end{split}
\end{equation*}

\item If $a_1/b_1< 2^\alpha-1$, then $t^*>1$. Hence,
$$
p_{\alpha,1}(x_1,x_2,x_3)\geq\frac{f(t^*)}{|b|^{1+\alpha}|a+b|^{1+\alpha}}.
$$

Since
$$f(t^*)=b_1^2\left(1+\frac{a_1}{b_1}\right)\left(1-\frac{a_1}{\left((a_1+b_1)^{1/\alpha}-b_1^{1/\alpha}\right)^\alpha}\right),$$
then
\begin{equation*}
\begin{split}
 f(t^*)&\geq b_1^2\min_{a_1<b_1(2^{\alpha}-1)}\left(1-\frac{a_1}{\left((a_1+b_1)^{1/\alpha}-b_1^{1/\alpha}\right)^\alpha}\right)\\*[7pt]
 &=b_1^2(2-2^{\alpha})\geq\frac{2-2^{\alpha}}{2^{2\alpha}}(a_1+b_1)^2,
\end{split}
\end{equation*}
since the function
$$
g(x)=1-\frac{x}{\left((x+b_1)^{1/\alpha}-b_1^{1/\alpha}\right)^{\alpha}}
$$
is decreasing and $(a_1+b_1)^2\leq (2^\alpha b_1)^2$.
\end{enumerate}
\end{itemize}
Now, If $x_1^1=x_2^1=x_3^1$, then $a_1=b_1=0$. Hence \eqref{alfa}
gives us $p_{\alpha,1}(x_1,x_2,x_3)=0$. On the other hand,
if $p_{\alpha,1}(x_1,x_2,x_3)=0$, inequality \eqref{projections}
gives us
$\max((x_2^i-x_1^i)^2,(x_3^i-x_2^i)^2,(x_3^i-x_1^i)^2)=0$, hence
$a_1^2=b_1^2=(a_1+b_1)^2=0$, which implies $x_1^1=x_2^1=x_3^1$.
\end{proof}

We are now ready to prove Proposition \ref{alphano}. Take a compact
subset $E$ of the $x_1$-axis with positive finite
$\alpha$-dimensional Hausdorff measure. Then by \cite [Theorem
1.1]{laura1}, $\Gamma_\alpha(E)=0$. It remains to show that
$\Gamma_{\alpha, \hat 1}(E)>0$. For this let $\mu$ be the
$\alpha$-dimensional Hausdorff measure restricted to $E$. Choosing
appropriately $E$ we can assume in addition that $\mu$ satisfies the
Ahlfors regularity condition $\mu(B(x,r)) \simeq r^\alpha$, \;\; $0
< r <  \text{diam}(E)$. In particular, $\mu$ has growth $\alpha$ and
is
 doubling.
It is enough to show that the singular integral operator
$R_\alpha^i$ associated with the scalar kernel
$k_\alpha^i=x_i/|x|^{1+\alpha}$, $i \neq 1$, is bounded on
$L^2(\mu)$.  This reduction is possible because the Davie-Oksendal
Lemma extends straightforwardly to several operators \cite[Lemma
4.2]{mattilaparamonov}. By Lemma \ref{permutalfa} we have
$p_{\alpha,\,i}(x_1,x_2,x_3)=0$ for $x_1, x_2$ and $x_3$ in $E$ and
$i \neq 1 $ and thus \eqref{mevescalar} yields
$$
\int|R_{\alpha,\, \ep}^i(\mu)(x)|^2\,d\mu(x) \le C\, \|\mu\|, \quad
\epsilon>0 .
$$
Replacing in the above inequality $\mu$ by $\chi_B\,\mu$ where $B$
is any ball we get
$$
\int_B |R_{\alpha,\, \ep}^i(\chi_B \,\mu)(x)|^2\,d\mu(x) \le C\,
\mu(B), \quad \epsilon>0.
$$
By the standard $T(1)$-Theorem of \cite{dajourne} we conclude that
$R_\alpha^i$ is bounded on~$L^2(\mu)$. \qed

\subsection{ Lipschitz harmonic capacity is not comparable to the capacity associated with a scalar Riesz-potential}
\label{counterexamplelipschitz} Inequality \eqref{mark} says that in
the plane, analytic capacity can be characterized in terms of either
capacity $\kappa_i$, $i=1,2$.  In particular this implies a weaker
qualitative statement, namely, that if $E$ is a compact set in the
plane and there exists a non-zero distribution $T$ supported on $E$
with  bounded potential $x_i/|x|^2 * T$, for $i=1$ or $i=2$, then
there exists another non-zero distribution $S$ supported on~$E$ with
bounded potentials $x_i/|x|^2 * S$, $i=1,2$.

In $\Rn$ Lipschitz harmonic capacity is an excellent replacement for
analytic capacity. Thus one may ask whether Lipschitz harmonic
capacity can be described in terms of one of the capacities
associated with a component of the kernel $x/|x|^{n}$ in which the
growth condition $n-1$
has been required on the distributions involved. In a qualitative
way we ask the following question. Assume that $E$ is a compact set
in $\Rn$ and that there exists a non-zero distribution $T$ supported
on $E$ with growth $n-1$ and bounded potential $x_n/|x|^n * T$. Is
it true that there exists another non-zero distribution $S$
supported on $E$ with bounded vector valued potential $x/|x|^n * T$
? The answer is no for $n \ge 3$. We describe the example in
${\mathbb R}^3 $.

\begin{prop}
 There exists a compact set $E\subset\mathbb{R}^3$ which supports a non-zero distribution $T$ with growth $2$ and bounded
scalar Riesz potential $x_3/|x|^3 *T$, but does not support any
non-zero distribution $S$ with bounded vector valued Riesz potential
$x/|x|^3 *S$.
\end{prop}

\begin{proof}
Let $K\subset H=\{(x_1,x_2,x_3)\in\mathbb{R}^3: x_3=0\}$ be the classical
$1$-dimensional planar Cantor set defined by taking the ``corner
quarters" at each generation. Then $K$ has finite positive length
but zero analytic capacity (see \cite{garnett},
\cite{garnettcantor} or \cite{ivanov}). In particular, $K$ has
zero Lipschitz harmonic capacity and by \cite{mattilaparamonov}
the same happens to $E=K\times [-1,1]$. Thus $E$ does not
support any distribution $S$ with bounded vector valued Riesz
potential $x/|x|^3 *S$.

Let $\mu$ denote $2$-dimensional Hausdorff measure restricted to
$K \times \mathbb{R} \subset {\mathbb{R}}^3$ and let $\nu$ denote
the restriction of $\mu$ to $E$.  It is a simple matter to check
that $\mu$ satisfies the growth condition
$$
\mu(B(x,r)) \le C\,r^2,\quad x \in K \times \mathbb{R},\quad 0
< r .
$$
Although the reverse inequality does not hold for large $r$,
$\mu$ is a doubling measure. Indeed, $\mu(B(x,r))$ is comparable
to $r^2$ for $0 < r \leq 1$ and to $r$ for $1 \leq r$. Our goal
is to show that the scalar Riesz singular integral operator $R^3$
with kernel $k^3(x)= x_3/|x|^3$ is bounded on $L^2(\nu)$. Once
this is established the Davie-Oksendal lemma (see \cite[Theorem 33
]{llibrechrist} or \cite[Lemma 4.2]{verdera}) provides a
non-negative function $b \in L^\infty(\nu)$ such that $x_3/|x|^3*
b \nu $ is in $L^\infty({\mathbb{R}}^3)\,,$ which completes the
proof.

The boundedness of $R^3$ on $L^2(\nu)$ follows directly from the boundedness of $R^3$ on $L^2(\mu)$. To show
this we check that $R^3(1)=0$ and then we apply the standard
$T(1)$-Theorem for doubling measures (see \cite{dajourne}). The computation of $R^3(1)$
is performed as follows. Set $K(x,\epsilon)= \{(y_1,y_2)\in K :
|x_1-y_1|>\epsilon \quad \text{and}\quad |x_2-y_2|>\epsilon \}$,
Then
\begin{equation*}
\begin{split}
R^3(1)(x)&= \lim_{\epsilon\rightarrow 0} \int_{|y-x|> \epsilon}
\frac{x_3-y_3}{|x-y|^3}\,d\mu(y)\\*[3pt]
&= \lim_{\epsilon\rightarrow 0}
\int_{K(x,\epsilon)} \left(\int_{|y_3-x_3|> \epsilon}
\frac{x_3-y_3}{|x-y|^3}\,dy_3 \right)\,dH^1(y_1,y_2) = 0,
\end{split}
\end{equation*}
for each $x \in K \times \mathbb{R}$.
\end{proof}

\subsection*{Remarks}
\begin{itemize}
\item Notice that  in the above
example one obtains that $R^3$ is bounded on $L^2(\nu)$, while the
whole vector $R$ is not bounded on $L^2(\nu)$. Therefore, the above
example shows that corollary \ref{l2bound} does not hold if $n\geq
3$, namely, we cannot get $L^2(\nu)$ boundedness of the vector
valued Riesz operator $R_{n-1}$ associated with a Riesz kernel of
homogeneity $-(n-1) $ from $L^2(\nu)$ boundedness of only one
component $R^i_{n-1}$.

\item It is an open question  to decide whether, for $n\geq 3$, Lipschitz
harmonic capacity is comparable to the capacities associated with
$(n-1)$-components of the vector valued Riesz potential $x/|x|^n
* T$.
\end{itemize}

\subsection{Finiteness of the capacities \boldmath$\kappa_i$}
\label{finiteness}

Indeed, we give a proof of a more general result, stating that for
compact sets $E\subset\Rn$, $0<\alpha<n$ and $1\leq i\leq n$, the
capacities
$$\kappa_{\alpha,\, i}(E)=\sup\left\{|\langle T,1\rangle|:\operatorname{spt}(T)\subset E,\,\left\|\frac{x_i}{|x|^{1+\alpha}}*T\right\|_\infty\leq 1\right\},
$$
are finite.

\begin{prop}\label{growth}
For any cube $Q\subset \Rn$, $0<\alpha<n$ and $1\leq i\leq n$, we
have $$\kappa_{\alpha,\,i}(Q)\leq Cl(Q)^\alpha.$$
\end{prop}

\begin{proof}
 Without loss of generality assume $i=1$. Assume also
momentarily that the dimension $n$ is odd, say $n=2k+1$.  Our
argument uses a reproduction formula for test functions involving
the kernel $k^i(y)=y_i/|y|^{1+\alpha}$, $1\leq i\leq n$, \cite
[Lemma 3.1]{laura1}. For a test function $g$, the formula reads
\begin{equation}\label{reproductionformula}
g(x)=c_{n,\alpha}\sum_{j=1}^n\left (\Delta^k\partial_j
g*\frac{1}{|y|^{n-\alpha}}*k^j\right)(x),
\end{equation}
for some constant $c_{n,\alpha}$ depending only on the dimension $n$
and on $\alpha$. For $n=2k$, there is an analogous reproduction
formula that settles the even case \cite[Lemma~3.1]{laura1}.

Let $T$ be a real distribution supported on $Q$ such that $k^1*T\in
L^{\infty}(\Rn)$. Write the cube $Q$ as $Q=I_1\times Q'$, with $I_1$
being an interval in $\mathbb{R}$ and $Q'$ an $n-1$ dimensional cube
in $\mathbb{R}^{n-1}$, and let $\varphi_Q \in {\mathcal
C}^{\infty}_0(2Q)$ be such that $\|\partial^s\varphi_Q\|_\infty\leq
C_sl(Q)^{-|s|}$ and
$$\varphi_Q(x)=\varphi_1(x_1)\varphi_2(x_2,\dotsc,x_n)$$
with $\varphi_1(x_1)=1$ on $I_1$,  $\varphi_1(x_1)=0$ on $(2I_1)^c$
and $\int^\infty_{-\infty}\varphi_1=0$, and $\varphi_2\geq 0$,
$\varphi_2\equiv 1$ on~$Q'$ and  $\varphi_2\equiv  0$ on $(2Q')^c$.
Then, since our distribution $T$ is supported on $Q$, using the
reproduction formula \eqref{reproductionformula},
\begin{equation*}
\begin{split}
|\langle T,1\rangle |&=|\langle T,\varphi_Q\rangle |\leq C\sum_{j=1}^n \left|\left\langle T,
\Delta^k\partial_j \varphi_Q*\frac{1}{|y|^{n-\alpha}}*k^j
\right\rangle \right|\\*[7pt]
&= C\left|\left\langle k^1*T, \Delta^k\partial_1
\varphi_Q*\frac{1}{|y|^{n-\alpha}}\right\rangle \right|+ C\sum_{j=2}^n \left|\left\langle T,
\Delta^k\partial_j \varphi_Q*\frac{1}{|y|^{n-\alpha}}*k^j
\right\rangle\right|\\*[7pt]
&=A+B.
\end{split}
\end{equation*}
We first estimate the term $A$. We have
\begin{multline*}
\int (k^1*T)(x)\, \Delta^k\partial_1
\varphi_Q*\frac{1}{|y|^{n-\alpha}}(x)\,dx = \int_{3Q} (k^1*T)(x)\,
(\Delta^k\partial_1 \varphi_Q*\frac{1}{|y|^{n-\alpha}})(x)\,dx
\\*[7pt]
+\int_{\Rn\setminus 3Q} (k^1*T)(x)\, (
\varphi_Q* \Delta^k\partial_1(\frac{1}{|y|^{n-\alpha}}))(x)\,dx.
\end{multline*}
 Let $Q_0$ be the unit cube centered at~$0$. Dilating to bring the integrals on $3Q_0$ and $2Q_0$, and using $|\partial^s\varphi_Q |\leq C_sl(Q)^{-|s|}$, we get
\begin{equation*}
\begin{split}
A&\leq
\|k^1*T\|_\infty\left(\int_{3Q}\int_{2Q}\frac{|\Delta^k\partial_1\varphi_Q(y)|}{|x-y|^{n-\alpha}}\,dy\,dx+ \int_{\Rn\setminus
3Q}\int_{2Q}\frac{|\varphi_Q(y)|}{|x-y|^{2n-\alpha}}\,dy\,dx\right)\\*[7pt]
&\leq Cl(Q)^{\alpha}\left(\int_{3Q_0}\int_{2Q_0}\frac{dy\,dx}{|x-y|^{n-\alpha}}+
\int_{\Rn\setminus
3Q_0}\int_{2Q_0}\frac{dy\,dx}{|x-y|^{2n-\alpha}}\right)\\*[7pt]
&\leq Cl(Q)^{\alpha}.
\end{split}
\end{equation*}

We turn now to the estimate of $B\,.$ The homogeneous differential
operator $\Delta^k$ can be written as $\Delta^k = \sum_{|s|=2k} a_s
\,\partial^s$, for certain constants~$a_s$. Divide the set of
multi-indexes~$s$ of length~$2k$ into two classes~$I$ and~$J$
according to whether $s_1 \ge 1$ or $s_1=0$. In other words, $s
\in I$ if $\partial^s$ contains at least one partial derivative with
respect to first variable. Thus $\Delta^k = \sum_{s \in I} a_s
\,\partial^s + \sum_{s \in J} a_s \,\partial^s $, and so
$B=B_1+B_2$ where
$$
B_1 = C\sum_{j=2}^n \left|\left\langle T,\sum_{s \in I} a_s \,\partial^s\partial_j
\varphi_Q*\frac{1}{|y|^{n-\alpha}}*k^j\right\rangle\right|
$$
and
$$
B_2 = C\sum_{j=2}^n \left|\left\langle T,\sum_{s \in J} a_s \,\partial^s\partial_j
\varphi_Q*\frac{1}{|y|^{n-\alpha}}*k^j\right\rangle\right| .
$$
To estimate $B_1$ we bring in each term of the sum in $s \in I$ one
derivative with respect to the first variable into the kernel $k^j$
and use $\partial_1 k^j= \partial_j k^1$ to take back a derivative
with respect to $j$ into $\varphi_Q$.  The effect of these moves
is to replace $k^j$ by $k^1$. Therefore
$$
B_1 = C\sum_{j=2}^n \left|\left\langle k^1*T,\sum_{|s|=2k} b_s \,\partial^s
\partial_j \varphi_Q*\frac{1}{|y|^{n-\alpha}}\right\rangle\right|,
$$
for some numbers $b_s$. This expression can be estimated as we did
before with $A$.

To estimate $B_2$ we need to replace in some way the kernel $k^j$ by
$k^1$. We do this by showing that, for each $j$ there exists a
function $\psi^j_Q \in \cc^\infty_0(2Q)$ satisfying
\begin{equation}\label{identity}
k^j*\varphi_Q=k^1*\psi^j_Q,\quad 1\leq j\leq n,
\end{equation}
and  $\|\partial^s\psi^j_Q\|_\infty\leq C_s l(Q)^{-|s|}$. Before
proving \eqref{identity} we show how to estimate $B_2$.

By \eqref{identity}
\begin{equation*}
\begin{split}
B_2 &= C\sum_{j=2}^n \left|\left \langle T,\sum_{s \in J} a_s \,\partial^s \partial_j
\varphi_Q*\frac{1}{|y|^{n-\alpha}}*k^j\right\rangle\right|\\*[7pt]
&=  C\sum_{j=2}^n
\left|\left\langle T,\sum_{s \in J} a_s \,\partial^s\partial_j
\psi^j_Q*\frac{1}{|y|^{n-\alpha}}*k^1\right\rangle\right|\\*[7pt]
& = C\sum_{j=2}^n
\left|\left\langle k^1*T,\sum_{s \in J} a_s \,\partial^s \partial_j
\psi^j_Q*\frac{1}{|y|^{n-\alpha}}\right\rangle\right|,
\end{split}
\end{equation*}
which can be estimated as the term $A$.

We are left with proving \eqref{identity}. Taking Fourier transforms
in~\eqref{identity} we obtain for some constant~$a$,
$$
a\,\hat{\varphi}_Q(\xi)\xi_j=\hat{\psi^j_Q}(\xi)\xi_1,
$$
which becomes
$$
a\,\partial_j\varphi_Q=\partial_1\psi^j_Q .
$$
Hence, for the non-trivial case $2 \le j\le n$,
$$
\psi^j_Q(x)=a\int_{-\infty}^{x_1}\partial_j\varphi_Q(t,x_2,\dotsc,x_n)\,dt=a\,\partial_j\varphi_2(x_2,\dotsc,x_n)\int_{-\infty}^{x_1}\varphi_1(t)\,dt,
$$
and the key remark is that the function above has support contained
in $2Q$ because the integral of $\varphi_1$ on the real line
vanishes.
\end{proof}

We conclude with the following corollary.

\begin{co}
For any compact set $E\subset\Rn$, $0<\alpha<n$ and $1\leq i\leq n$,
we have $\kappa_{\alpha,\,i}(E)\leq C\operatorname{diam}(E)^\alpha$.
\end{co}

When $n=2$ and $\alpha=1$, \eqref{mark} implies that $\kappa_i(E)\le C M^1(E)$, $i=1,2$, where $M$ stands for the one dimensional Hausdorff content.
In general, we do not know whether in the preceding inequality the diameter of
$E$ can be replaced by the $\alpha-$dimensional Hausdorff content of $E$.

\subsection{Localization and growth}
The growth assumption on the distribution $T$ in the localization
lemma (Lemma \ref{localization1}) cannot be completely dispensed
with. Indeed, if $x_i/|x|^2 \in L^\infty(\mathbb{R}^n)$ and one
has the inequality
\begin{equation}
\left\|\frac{x_i}{|x|^2}*\varphi_Q T\right\|_\infty\leq C\left\|
\frac{x_i}{|x|^2}*T\right\|_\infty,
\end{equation}
for all $\varphi_Q \in \cc^\infty_0(Q)$ satisfying the normalization
condition \eqref{normalization},
 then necessarily $T$ has linear growth. This
can be shown by an argument very close to that of the previous
subsection. We only deal with the details of the case $n=2$. The
case of even dimensions is very similar, while the case of odd
dimensions needs some additional care. We also assume $i=1$.

Let $Q$ be square and $\varphi_Q$ a function in $\cc^\infty_0(Q)$
satisfying the normalization condition \eqref{normalization}.
Set
$Q=I_1 \times I_2$ and $\psi(x_1,x_2)= \psi_1(x_1)\psi(x_2)$,
where, for $j=1,2$, $\psi_j \in \cc^\infty_0(I_j)$, $\psi_j =
1$ on $I_j$, $\int^\infty_{-\infty} \psi(x_1) \,dx_1 =0$ and $\|
d^k \psi_j/(dx_j)^k \|_\infty \le C\,l(I_j)^{-k}$, $0 \le k\le
2$. We then have
$$
\langle{T, \varphi_Q}\rangle = \langle{\varphi_Q\,T,1}\rangle=
\langle{\varphi_Q\,T,\psi}\rangle.
$$
We want now to find a function $\chi$ such that $\psi=
k^1*\chi$, where $k^1 = x_1/|x|^2$. Taking the Fourier
transform we get $\hat \psi(\xi) = a(\xi_1/|\xi|^2)\,\hat{\chi}
(\xi)$ for some constant $a\,.$ Hence $\partial_1 \chi =
b\,\Delta\psi$, for some other constant $b$. Thus
\begin{equation*}
\begin{split}
\chi &=b\int_{-\infty}^{x_1} \Delta\psi(t,x_2)\,dt \\*[7pt]
&=b\left(\partial_1 \psi_1(x_1)\,
\psi_2(x_2)+\left(\int_{-\infty}^{x_1} \psi_1(t)\,dt
\right)\,\partial_2^2 \psi_2(x_2)\right).
\end{split}
\end{equation*}
Notice that $\chi$ is supported on $Q$ and $\|\chi\|_\infty \le
C\,l(Q)^{-1}$. Therefore 
$$
|\langle{T, \varphi_Q}\rangle| = |\langle{k^1* \varphi_Q
T,\chi}\rangle| \le C\,\|k^1* \varphi_QT\|_\infty
\,\|\chi\|_{L^1(Q)} \le C\,l(Q).\rlap{\hspace*{1.26cm}\qed}
$$

\subsection{The growth condition for positive measures}
We start by showing that the usual linear growth condition for a
positive Radon measure is equivalent to the linear growth
condition for distributions as defined in \eqref{growthG}. Later
on we treat also the case of the $\alpha$-growth condition for $0<
\alpha < n.$

Given a positive Radon measure $\mu$ set
$$
L(\mu)= \underset{Q} \sup \frac{\mu(Q)}{l(Q)},
$$
where the supremum is taken over all cubes $Q$ with sides parallel
to the coordinate axis.

If $\varphi \in \cc^\infty_0(\Rn),$ then by an inequality of Mazya
\cite[1.2.2, p.
24]{mazya}
$$
|\langle \mu, \varphi \rangle | = |\int \varphi\,d\mu|\le \int
|\varphi|\,d\mu \le C\, L(\mu)\,\int |\nabla^{n-1}\varphi
(x)|\,dx,
$$
where $\nabla^{n-1}\varphi$ denotes, as usual , the vector of all
derivatives $\partial^s \varphi$ of order $|s|=n-1$. Thus
$$
G(\mu) \le C\,L(\mu).
$$
The reverse inequality is immediate. Indeed, given a cube $Q$ let
$\varphi_Q$ be a function in $\cc^\infty_0(2Q)$ such that $1 \le
\varphi_Q$ on $Q$ and $\|\partial^s\varphi_Q\|_\infty \le C_s
\,l(Q)^{-|s|}, \;|s|\ge 0.$ Then
$$
\mu(Q) \le \int \varphi_Q \,d\mu = |\langle \mu, \varphi_Q \rangle |
\le C\,G(\mu)\,l(Q)
$$
because $ C_s^{-1}\, l(Q)^{-1}\,\partial^s \varphi $ is an atom for
$|s|=n-1$, and so $\|\partial^s \varphi\|_{H^1(\Rn)} \le C\,l(Q),
|s|= n-1 .$

We proceed now to treat the case of a general $\alpha$-growth
condition, $0 < \alpha  < n $. Set
$$
L_\alpha(\mu)= \underset{Q} \sup \frac{\mu(Q)}{l(Q)^\alpha},
$$
where the supremum is taken over all cubes $Q$ with sides parallel
to the coordinate axis. We consider first the inequality $
L_\alpha(\mu) \le C\, G_\alpha(\mu)$. The definition of $G_\alpha$
is in \eqref{growthalfa}. Given a cube $Q$ let $\varphi_Q$ be a
function in $\cc^\infty_0(2Q)$ such that $1 \le \varphi_Q$ on $Q$
and $\|\partial^s\varphi_Q\|_\infty \le C_s \,l(Q)^{-|s|}, \;|s|\ge
0.$ We claim that $c\,\varphi_Q$ satisfies the normalization
inequalities \eqref{integerh1} or \eqref{alfah1} for a sufficiently
small positive constant $c$. If this is the case, then
$$\mu(Q) \le \int \varphi_Q \,d\mu = |\langle \mu, \varphi_Q \rangle |\le c^{-1}\,G_\alpha(\mu)\,l(Q).$$

We treat first the case of integer $\alpha$. Clearly $\|\partial^s\varphi_Q\|_{L^1}\le Cl(Q)^{\alpha}$, $|s|=n-\alpha$.
By H\"older's inequality and the fact that Riesz transforms preserve $L^q(\Rn)$, $1<q<\infty$,
$$\|R_j(\partial^s\varphi_Q)\|_{L^1(4Q)}\le Cl(Q)^{\frac n p}\|R_j(\partial^s\varphi_Q)\|_{L^q(\Rn)}
\le Cl(Q)^{\frac n p}\|\partial^s\varphi_Q\|_{L^q(\Rn)}\le
Cl(Q)^\alpha.$$ Then, by the Sublemma in subsection 4.1,
 the function $\varphi_Q$ satisfies the normalization inequalities \eqref{integerh1}.

If $\alpha\notin\Z$, write $\alpha=[\alpha]+\{\alpha\}$, with $[\alpha]\in\Z$ and $0<\{\alpha\}<1$.
  For the claim we have to show that for $|s|=n-[\alpha],$ and $1\le j\le n$,
  \begin{equation}\label{u}\|\partial^s\varphi_Q*\frac 1{|x|^{n-\{\alpha\}}}\|_{L^1(\Rn)}\le l(Q)^\alpha\end{equation}
  \begin{equation}\label{dos}\|R_j(\partial^s\varphi_Q*\frac 1{|x|^{n-\{\alpha\}}})\|_{L^1(\Rn)}\le l(Q)^\alpha.\end{equation}

\noindent Inequality \eqref{u} is proven as follows. By Fubini, $$\int_{4Q}|(\partial^s\varphi_Q*\frac 1 {|x|^{n-\{\alpha\}}})(x)|dx\le Cl(Q)^{\alpha}.$$
 As in the Sublemma, integrating by parts to take one derivative from $\partial^s\varphi_Q$ to the kernel $1 /|x|^{n-\{\alpha\}}$ we obtain
$$\int_{(4Q)^c}|(\partial^s\varphi_Q*\frac 1 {|x|^{n-\{\alpha\}}})(x)|dx\le Cl(Q)^{\alpha},$$ which proves \eqref{u}.

The prove of inequality \eqref{dos} we use that, for some constant
$c=c(n,\alpha)$,
$$R_j(\partial^s\varphi_Q *\frac{1}{|x|^{n-\{\alpha\}}})=c\;\partial^s\varphi_Q*\frac{x_j}{|x|^{n+1-\{\alpha\}}}.$$
This can be easily checked by taking the Fourier transform. Now the
argument described above to prove \eqref{u} applies with small
changes to prove \eqref{dos}.

For the reverse inequality, namely  $ G_\alpha(\mu) \le C\,
L_\alpha(\mu)$, it is convenient to distinguish two
cases.
\begin{itemize}
\item {\sl $\alpha$ is integer.}
The argument is exactly as in the case $\alpha=1$. If $\varphi \in
\cc^\infty_0(\Rn),$ then by an inequality of Mazya \cite[1.2.2, p.
24]{mazya}
$$
|\langle \mu, \varphi \rangle | = |\int \varphi\,d\mu|\le \int
|\varphi|\,d\mu \le C\, L_\alpha(\mu)\,\int
|\nabla^{n-[\alpha]}\varphi (x)|\,dx.
$$
 Thus
$$
G_{\alpha}(\mu) \le C\,L_{\alpha}(\mu).
$$

\item {\sl $\alpha$ is not integer.} If $\varphi \in
\cc^\infty_0(\Rn),$ then by another inequality of Mazya \cite[3.4.1,
p. 134]{mazya}
$$
|\langle \mu, \varphi \rangle | = |\int \varphi\,d\mu|\le \int
|\varphi|\,d\mu \le C\, L_\alpha(\mu)\,\int
|\nabla^{n-[\alpha]}\varphi (x)*\frac 1{|x|^{n-\{\alpha\}}}|\,dx.
$$
 Thus
$$
G_{\alpha}(\mu) \le C\,L_{\alpha}(\mu).
$$

\end{itemize}

\begin{gracies}
The authors are indebted to X.Tolsa for several useful conversations,
which led to the example in section 6.2 and to M. Melnikov for reminding them the well known argument to obtain the first inequality in \eqref{mark}. The authors were partially
supported by grants 2009-SGR-420 (Generalitat de Catalunya) and
MTM2007-60062 (Spanish Ministry of Science).
\end{gracies}

\noindent Departament de Matem\`atiques, Universitat Aut\`onoma de
Barcelona, 08193 Bellaterra (Barcelona), Catalonia.\newline\newline
{\em E-mail:} {\tt mateu@mat.uab.cat}, {\tt laurapb@mat.uab.cat},
{\tt jvm@mat.uab.cat}.

\end{document}